\documentclass[reqno,onecolumn,oneside]{paper}
\usepackage[english]{babel}
\usepackage{enumerate,amsmath,amsthm,amsfonts,pifont,slashed,ifsym,yfonts,calligra,amssymb,latexsym,mathrsfs}
\usepackage[colorlinks]{hyperref}
\usepackage[all]{xy}
\usepackage[sort]{cite}

\newtheorem{theorem}{Theorem}[section]
\newtheorem{proposition}[theorem]{Proposition}
\newtheorem{lemma}[theorem]{Lemma}
\newtheorem{corollary}[theorem]{Corollary}
\theoremstyle{definition}
\newtheorem{definition}[theorem]{Definition}
\newtheorem{exm}[theorem]{Example}

\theoremstyle{remark}
\newtheorem{remark}[theorem]{Remark}

\newcommand{\bydef}{\mathrel{\mathop:}=}
\newcommand{\vr}{^{\vee\odot}}
\renewcommand{\wr}{^{\wedge\oplus}}

\newcommand{\End}{\operatorname{End}}

\newcommand{\id}{\operatorname{id}}

\newcommand{\frab}{\operatorname{Free}_{\Gab}}

\newcommand{\fr}{\operatorname{Free}}

\renewcommand{\P}{\operatorname{P_F}}
\newcommand{\supp}{\operatorname{supp}}

\newcommand{\gr}{Grothendieck }
\newcommand{\ID}{\operatorname{\textsc{Id}}}

\newcommand{\Gab}{\cat{G}^\text{Ab}}

\newcommand{\sL}{{{}^s\!\mathcal{L}}}
\newcommand{\sR}{{{}^s\!\mathcal{R}}}
\newcommand{\B}{\operatorname{B}}

\newcommand{\M}{\mathcal{M}}

\newcommand{\MV}{\mathcal{MV}}
\newcommand{\Mod}{\textrm{-}{}^s\!\!\mathcal{M}\!\!\:\mathit{od}}
\newcommand{\Modd}{{}^s\!\!\mathcal{M}\!\!\:\mathit{od}\textrm{-}}

\newcommand{\cat}{\mathcal}

\newcommand{\R}{\mathbb{R}}
\newcommand{\N}{\mathbb{N}}

\renewcommand{\wp}{\mathscr{P}}

\renewcommand{\vec}{\mathbf}

\renewcommand{\phi}{\varphi}

\newcommand{\g}{\gamma}

\newcommand{\la}{\langle}
\newcommand{\ra}{\rangle}

\newcommand{\lto}{\longrightarrow}

\newcommand{\lmapsto}{\longmapsto}

\newcommand{\ov}{\overline}

\newcommand{\tensor}{\otimes}

\newcommand{\0}{\mathbf{0}}

\newcommand{\z}{{\declareslashed{}{\sim}{0}{-0.2}{0}\slashed{0}}}
\renewcommand{\u}{{\declareslashed{}{\sim}{0}{-0.2}{1}\slashed{1}}}

\def\amslatex\slash{{\protect\AmS-\protect\LaTeX}}

\begin{document} 

\title{Semiring and semimodule issues in MV-algebras}

\author{Antonio Di Nola \and Ciro Russo} 

\institution{Dipartimento di Matematica \\ Universit\`a di Salerno, Italy \\ \small{\texttt{$\{$adinola,cirusso$\}$@unisa.it}}}


\maketitle
\today

\begin{abstract}
In this paper we propose a semiring-theoretic approach to MV-algebras based on the connection between such algebras and idempotent semirings established in \cite{dng} and improved in \cite{bdn} --- such an approach naturally imposing the introduction and study of a suitable corresponding class of semimodules, called MV-semimodules.

Besides some basic yet fundamental results of more general interest for semiring theory we present several results addressed toward a semiring theory for MV-algebras. In particular we give a representation of MV-algebras as a subsemiring of the endomorphism semiring of a semilattice, show how to construct the Grothendieck group of a semiring and prove that this construction has a functorial nature. We also study the effect of Mundici categorical equivalence between MV-algebras and lattice-ordered Abelian groups with a distinguished strong order unit \cite{mun} upon the relationship between MV-semimodules and semimodules over idempotent semifields.
\end{abstract}

\keywords{MV-algebra \ Semiring \ Semimodule}

\noindent
{\it 2000 Mathematics Subject Classification:} 06D35, 16Y60

\vskip 10pt

\noindent The authors wish to thank the anonymous referee for the huge amount of suggestions that greatly improved the presentation of the paper.

\section{Introduction}

The main objective of this paper is to continue and deepen the study of MV-algebras as a special class of idempotent semirings proposed in \cite{dng} by applying classical ring-theoretic constructions and techniques whose adaptability to semiring theory is either known or established here.

MV-algebras arose in the literature as the algebraic semantics of \L ukasiewicz propositional logic, one of the longest-known many-valued logics. MV-algebras can be seen in one of their facets as a non-idempotent generalization of Boolean algebras and, among the various many-valued logics and corresponding algebraic semantics, \L ukasiewicz logic and MV-algebras are the ones that best succeed in both having a rich expressive power and preserving many properties of symmetry that are inborn qualities of classical propositional logic and Boolean algebras (for detailed discussions about these aspects of MV-algebras the reader may refer to \cite{dnr3,rus}).

In the last decades the knowledge about MV-algebras benefited from the literature on lattice-ordered groups via the well-known and celebrated categorical equivalence between MV-algebras and lattice-ordered Abelian groups with a distinguished strong order unit (Abelian $u\ell$-groups for short) \cite{mun}.

On the other hand, the theory of idempotent semirings is nowadays well-established (see for instance \cite{golan,kat1,kat4,kat6,kat7,kat8}) and boasts a wide range of applications in many fields, such as discrete mathematics, computer science, computer languages, linguistic problems, finite automata, optimization problems, discrete event systems, computational problems et cetera (see, for instance, \cite{bac,car,co0,gla,gon1,gon2,gun,kol,kol2,litmas}). The theory arising from the substitution of the fields of real and complex numbers with idempotent semirings and/or semifields is often referred to as \emph{idempotent} or \emph{tropical mathematics}.

As Litvinov observed in \cite{lit2}, ``idempotent mathematics can be treated as the result of a dequantization of the traditional mathematics over numerical fields as the Planck constant $\hbar$ tends to zero taking imaginary values.'' This point of view was also presented, by Litvinov himself and Maslov, in \cite{litmas}. Another equivalent presentation of idempotent mathematics is as an asymptotic version of the traditional mathematics over the fields of real and complex numbers. This idea is expressed in terms of an idempotent correspondence principle which is closely related to the well-known correspondence principle of N. Bohr in quantum theory \cite{bohr}. In fact, many important and useful constructions and results of the traditional mathematics over fields correspond to analogous constructions and results over idempotent semirings and semifields; to this extent the aforementioned paper \cite{lit2} provides an impressive list of references.

Another important aspect of the development of such a theory is the linear algebra and the algebraic geometry of idempotent semirings, better known as \emph{tropical geometry}, whose most important model is the geometry of the tropical semifield $\la \ov\R, \min, +, \infty, 0\ra$, where $\ov \R = \R \cup \{\infty\}$; the key objects of study are polyhedral cell complexes which behave like complex algebraic varieties. In this area, relevant works the reader may refer to are, among others, \cite{co1,co2,gun,rich}.

A connection between MV-algebras and a special category of additively idempotent semirings (called \emph{MV-semirings} or \emph{\L ukasiewicz semirings}) was first observed in \cite{dng} and eventually enforced in \cite{bdn}. On the one hand, every MV-algebra has two \emph{semiring reducts} isomorphic to each other by the involutive unary operation $^*$ of MV-algebras (see Section \ref{mvsemirings} for the definition of MV-algebra); on the other hand, the category of \emph{MV-semirings} defined in~\cite{bdn} is isomorphic to the one of MV-algebras. Such results led to interesting applications of MV-semirings and their semimodules to the theory of fuzzy weighted automata~\cite{sch}, and to an algebraic approach to fuzzy compression algorithms \cite{dnr,dnr2} and mathematical morphological operators \cite{rus2} for digital images.

Another link between MV-algebras and semiring theory relies on the aforementioned categorical equivalence due to Mundici. Indeed we shall see later on in the paper that the category of Abelian $u\ell$-groups is isomorphic to the one of idempotent semifields with a distinguished strong order unit (idempotent $u$-semifields for short); this fact has interesting consequences on the categories of semimodules over a given MV-algebra and the idempotent $u$-semifield corresponding to it via the composition of Mundici equivalence and such a categorical isomorphism.

In the present paper, after a preliminary section in which we recall basic notions and results about semirings and semimodules and a section devoted to projective objects in the categories of semimodules, we shall focus our attention on idempotent semirings and particularly on MV-semirings. Our main results can be briefly summarized as follows.
\begin{itemize}
\item We present a representation for homomorphisms of free semimodules (Theorem~\ref{homofree}) which, for finitely generated ones, is completely analogous to the matrix representation of homomorphisms between free ring modules. Such a representation is then suitably extended to all semimodule homomorphisms (Theorem~\ref{homo}).
\item A matrix-based characterization of finitely generated projective semimodules over any semiring is proved in Theorem~\ref{finproj}. Also in this case the result is a plain generalization of the corresponding one in ring module theory.
\item In Proposition~\ref{cyclicmv} we characterize cyclic projective MV-semimodules as direct summands of the free cyclic one. It is worth to underline that such a characterization does not hold in general for idempotent semirings. Moreover, it is not known so far whether it can be extended to non-cyclic MV-semimodules.
\item Corollary~\ref{mvrepr1} is a representation of any MV-algebra as a subsemiring of a semiring of endomorphisms. An interesting aspect of this result is the fact that the non-idempotent MV-algebraic operations $\oplus$ and $\odot$, which are commutative, are represented as composition of endomorphisms, an operation which is typically non-commutative.
\item In Section~\ref{gro}, we construct the \gr groups of semirings and MV-algebras following the classical ring-theoretic construction; such a construction is easily proven, also thanks to Theorem~\ref{finproj}, to be functorial (Theorems~\ref{k0thm} and \ref{k0thmsr}).
\item Last, in Section~\ref{functor} we discuss the relationship between MV-semimodules and semimodules over the positive cones of idempotent $u$-semifields as a consequence both of Mundici categorical equivalence and of constructions and results, of more general interest for idempotent semirings, that we present and/or recall in the same section.

In particular, we shall see that the category of semimodules over a given MV-algebra $A$ is basically a full subcategory of the one of semimodules over the positive cone of the idempotent $u$-semifield corresponding to $A$ via Mundici functor (Corollary \ref{mvemb}). Moreover, we shall see that any MV-semimodule is in a suitable sense a sort of interval of a semimodule over the positive cone of an idempotent $u$-semifield in analogy with the construction via Mundici functor of any MV-algebra as an interval of an Abelian $u\ell$-group with suitably defined operations. Consequently, MV-semimodules can be seen as ``truncated'' semimodules over positive cones of idempotent semifields, and global properties of the former as local ones of the latter.
\end{itemize} 

Many different categories are mentioned and used in this paper; although objects and morphisms of some of them shall be explicitly defined later we introduce the notations we use throughout the work in the following table.

\begin{center}
	\begin{tabular}{ccc}
Category & Objects & Morphisms \\ \hline
$\sR$ & Semirings with identity & Identity-preserving \\
&&semiring homomorphisms\\
&&\\
$S\Mod$ & Left $S$-semimodules & Left $S$-semimodule homomorphisms \\ 
&&\\
$\Modd S$ & Right $S$-semimodules & Right $S$-semimodule homomorphisms \\ 
&&\\
$\M$ & Monoids & \\
$\cat M^{\text{Ab}}$ & Abelian monoids & Monoid homomorphisms  \\
$\sL$ & Semilattices with identity & \\
&&\\
$\Gab$ & Abelian groups & Group homomorphisms \\
&&\\
 & Lattice-ordered Abelian & Lattice-ordered group \\
$\ell\Gab_u$ & groups with a distinguished& homomorphisms that \\
 & strong order unit  & preserve the strong unit \\
&&\\
$\MV$ & MV-algebras & MV-algebra homomorphisms\\		
	\end{tabular}
\end{center}

\section{Semirings and Semimodules}
\label{srsm}

In this section we recall some basic definitions and properties of semirings and semimodules over them. Most of this material can be found in~\cite{golan}. 

\begin{definition}\label{semiring}
A \emph{semiring} is an algebraic structure $\la S, +, \cdot, 0, 1 \ra$ such that
\begin{enumerate}[(S1)]
\item $\la S, +, 0\ra$ is a commutative monoid;
\item $\la S, \cdot, 1\ra$ is a monoid;
\item $\cdot$ distributes over $+$ from either side;
\item $0 \cdot a = 0 = a \cdot 0$ for all $a \in S$.
\end{enumerate}
A semiring $S$ is called
\begin{itemize}
\item \emph{commutative} if so is the multiplication,
\item \emph{idempotent} if so is the sum, i.~e. if it satisfies the equation $x + x = x$,
\item a \emph{semifield} if $\la S \setminus \{0\}, \cdot, 1 \ra$ is an Abelian group.
\end{itemize}
\end{definition}

Many relevant examples of semirings are known, among which we recall the (commutative) one of natural numbers $\la \N_0, +, \cdot, 0, 1\ra$ and the following ones, whose relevance will be clear in next sections.

\begin{exm}\label{end}
Let $\la M, +, 0\ra$ be a commutative monoid and $\End_\M(M)$ the set of its endomorphisms. Obviously $\la \End_\M(M), \circ, \id \ra$ is a monoid and $\la \End_\M(M), +, \0 \ra$, with the pointwise sum and the zero-constant map, is a commutative monoid. Now, if we consider three endomorphisms $f,g$ and $h$ of $M$, for all $x \in M$ we have:
\begin{enumerate}
\item[] $(h \circ (f+g))(x) = h(f(x) + g(x)) = h(f(x)) + h(g(x)) = ((h \circ f) + (h\circ g))(x)$, hence $h \circ (f+g) = (h \circ f) + (h \circ g)$;
\item[] $((f+g) \circ h)(x) = (f+g)(h(x)) = f(h(x)) + g(h(x)) = ((f \circ h) + (g \circ h))(x)$, that is $(f+g) \circ h = (f\circ h) + (g\circ h)$.
\end{enumerate}
Therefore $\la \End_\M(M), +, \circ, \0, \id \ra$ is a semiring, for any commutative monoid $M$. It can be shown analogously that also $\la \End_\M(M), +, \cdot, \0, \id \ra$, with $f\cdot g \bydef g \circ f$, is a semiring.
\end{exm}

\begin{exm}\label{l-group}
A partially ordered group $\la G, \cdot, {}^{-1}, 1, \leq\ra$ is a group endowed with an order relation which is compatible with the binary operation, i.~e., such that $a \leq b$ implies $ca \leq cb$ and $ac \leq bc$ for all $a,b,c \in G$. If the order relation defines a lattice structure, then the group is called a \emph{lattice-ordered group}, \emph{$\ell$-group} for short.

Let $\la G, \cdot, {}^{-1}, 1, \vee, \wedge\ra$ be an $\ell$-group and let us add a bottom element $\bot$ to $G$. If we set $x \cdot \bot = \bot = \bot \cdot x$ for all $x \in \ov G = G \cup \{\bot\}$, then the structure $\la \ov G, \vee, \cdot, {}^{-1}, \bot, 1 \ra$ is an idempotent division semiring. The same can be done by adding a top element and setting $\wedge$ instead of $\vee$ as the semiring sum; in this case we obtain the idempotent semifield $\la G \cup \{\top\}, \wedge, \cdot, {}^{-1}, \top, 1 \ra$. 

Conversely, let $\la F, \vee, \cdot, {}^{-1}, \bot, 1\ra$ be an idempotent division semiring. So $\la F \setminus \{\bot\}, \cdot, {}^{-1}, 1\ra$ is a group and the semilattice order defined by $\vee$ is compatible with $\cdot$. Moreover, it is immediate to verify that, for all $x, y \in F \setminus \{\bot\}$, $x \wedge y = -((-x) \vee(-y))$ and, therefore, $\la F \setminus \{\bot\}, \cdot, {}^{-1}, 1, \vee, \wedge\ra$ is a lattice-ordered group. 

We notice that the constructions above actually define a categorical isomorphism, and its inverse, between $\ell$-groups, with $\ell$-group homomorphisms, and idempotent division semirings with semiring homomorphisms.
\end{exm}

\begin{exm}\label{trop}
Let $\ov \R = \R \cup \{-\infty\}$. The structure $\la \ov \R, \max, +, -\infty, 0 \ra$ is an idempotent semifield, sometimes called the \emph{tropical semiring}.

According to the previous example, it is possible to define also the semifield $\la \R \cup \{\infty\}, \min, +, \infty, 0 \ra$. Indeed, in the literature of tropical geometry and idempotent semirings, both these structures are referred to as ``tropical semiring.'' Moreover, some authors call tropical semiring also the non-negative part of $\ov\R$. Actually, apart from notational convenience, there are no major differences among the theories which can be developed on all of these semirings.
\end{exm}

\begin{definition}\label{semimodule}
Let $S$ be a semiring. A (left) \emph{$S$-semimodule} is a commutative monoid $\la M, +, 0 \ra$ with an external operation with coefficients in $S$, called \emph{scalar multiplication}, $\cdot: (a,x) \in S \times M \lmapsto a \cdot x \in M$, such that the following conditions hold for all $a, b \in S$ and $x, y \in M$:
\begin{enumerate}[(SM1)]
\item $(a b) \cdot x = a \cdot (b \cdot x)$,
\item $a \cdot (x + y) = (a \cdot x) + (a \cdot y)$,
\item $(a + b) \cdot x = (a \cdot x) + (b \cdot x)$,
\item $0_S \cdot x = 0_M = a \cdot 0_M$,
\item $1 \cdot x = x$.
\end{enumerate}
\end{definition}

\begin{exm}\label{funcsemimodule}
Let $S$ be a semiring and $X$ be an arbitrary non-empty set. We can consider the monoid $\la S^X, +, \0 \ra$, where $\0$ is the $0_S$-constant function from $X$ to $S$ and
$$(f + g)(x) = f(x) + g(x) \quad \textrm{ for all } x \in X \textrm{ and } f, g \in S^X.$$
Then we can define a scalar multiplication in $S^X$ as follows:
$$\cdot: (a,f) \in S \times S^X \lmapsto a \cdot f \in S^X,$$
with the map $a \cdot f$ defined as $(a \cdot f)(x) = a f(x)$ for all $x \in X$.

It is clear that $S^X$ is a left $S$-semimodule. The semimodule $S^X$ can be defined also for $X = \varnothing$, in which case we obtain, up to an isomorphism, the one-element semimodule $\{0\}$.
\end{exm}

The definition and properties of right $S$-semimodules are completely analogous. If $S$ is commutative, the concepts of right and left $S$-semimodules coincide and we will say simply $S$-semimodules. If a monoid $M$ is both a left $S$-semimodule and a right $T$-semimodule~--- over two given semirings $S$ and $T$~--- we will say that $M$ is an \emph{$(S,T)$-bisemimodule} if the following associative law holds:
\begin{equation*}
(a \cdot_l x) \cdot_r a' = a \cdot_l (x \cdot_r a'), \quad \textrm{for all } \ x \in M, \ a \in S, \ a' \in T,
\end{equation*}
where $\cdot_l$ and $\cdot_r$ are~--- respectively~--- the left and right scalar multiplications.

\begin{definition}\label{smhomo}
Let $S$ be a semiring and $M, N$ be two left $S$-semimodules. A map $f: M \lto N$ is an $S$-semimodule homomorphism if $f(x + y) = f(x) + f(y)$ for any $x, y \in M$, and $f(a \cdot x) = a \cdot f(x)$, for all $a \in S$ and $x \in M$.
\end{definition}

Thus, given a semiring $S$, the categories $S\Mod$ and $\Modd S$ have, respectively, left and right $S$-semimodules as objects, and left and right $S$-semimodule homomorphisms as morphisms. If $S$ is commutative, $S\Mod$ and $\Modd S$ coincide, and we will simply use the left notation $S\Mod$.
\begin{remark}\label{notation}
Henceforth, in all the definitions and results that can be stated both for left and right semimodules, we will refer generically to ``semimodules''~--- without specifying left or right~--- and we will use the notations of left semimodules.
\end{remark}

\begin{exm}\label{n-monoid}
Any commutative monoid $M$ is naturally an $\N_0$-semimodule with the scalar multiplication defined as
$$nx = \underbrace{x + \ldots + x}_{n \textrm{ times}} \ \textrm{ and } \ 0x = 0, \quad \textrm{for all $n \in \N$ and $x \in M$}.$$
Moreover, it is immediate to verify that the categories of $\N_0$-semimodules and commutative monoids actually coincide for they have the same objects and morphisms.
\end{exm}

\begin{exm}\label{end-monoid}
With reference to Example~\ref{end}, any commutative monoid $M$ is trivially a right $\End_\M(M)$-semimodule with the action defined simply as the action of the endomorphisms on the elements of the monoid. But we can say more, as the next result shows.
\end{exm}

\begin{proposition}\label{semiend}
Let $M$ be a commutative monoid and $S$ be a semiring. Then $M$ can be endowed with an $S$-semimodule structure if and only if there exists a semiring homomorphism from $S$ to the semiring $\End_\M(M)$ of the monoid endomorphisms of $M$. 
\end{proposition}
\begin{proof}
First suppose that $M$ is a left $S$-semimodule and, for $a \in S$, let $h_a: x \in M \lmapsto a \cdot x \in M$. By (SM2) and (SM5) of Definition~\ref{semimodule}, $h_a$ is a monoid homomorphism, hence we have a map
\begin{equation}\label{xi}
\xi: \quad a \ \in \ S \quad \lmapsto \quad h_a \ \in \ \End_\M(M).
\end{equation}
By (SM4) of Definition~\ref{semimodule}, $\xi(0_S) = \0$, by (SM5) $\xi(1) = \id$, and the fact that $\xi$ preserves sums comes immediately from (SM3). Last, by (SM1), if $M$ is a left semimodule,  we have $h_{ab}(x) = (ab) \cdot x = a \cdot (b \cdot x) = h_a (h_b(x)) = (h_a \circ h_b)(x)$, for all $a, b \in S$ and $x \in M$. So $\xi$ is a semiring homomorphism from $S$ to $\la\End_\M(M), +, \circ, \0, \id_M\ra$. The proof for right semimodules is similar but the multiplication in $\End_\M(M)$ is the composition in the reverse order.

The converse implication is trivial. Indeed it is true in general that, if $h: S \lto T$ is a semiring homomorphism and $M$ is an $T$-semimodule, then $h$ induces an $S$-semimodule structure on $M$ defined by $a \cdot_h x = h(a) \cdot x$ for all $a \in S$ and $x \in M$.\footnote{This situation, in the case of idempotent semirings, will be discussed more in details in Section~\ref{functor}.} Here we have precisely this situation, with $T = \End_\M(M)$.
\end{proof}

\begin{corollary}\label{semirepr}
Let $S$ be a semiring and $S^+ = \la S, +, 0_S\ra$ its additive monoid reduct. Then $S$ can be embedded in the semiring $\End_\M(S^+)$.
\end{corollary}
\begin{proof}
$S$ is the free one-generated $S$-semimodule, hence there exists the semiring homomorphism $\xi$ defined by (\ref{xi}). If $a \neq b \in S$, $h_a(1) = a \cdot 1 = a \neq b = b \cdot 1 = h_b(1)$ whence $\xi(a) \neq \xi(b)$.
\end{proof}

Let $S$ be a semiring and $M$ an $S$-semimodule. An \emph{$S$-subsemimodule} $N$ of $M$ is a submonoid of $M$ which is stable with respect to the scalar multiplication. It is easy to verify that, for any family $\{N_i\}_{i \in I}$ of $S$-subsemimodules of $M$, $\left\la \bigcap_{i \in I} N_i, +, 0 \right\ra$ is still a subsemimodule of $M$. Thus, given an arbitrary subset $X$ of $M$, we define the $S$-subsemimodule $\la X \ra$ \emph{generated by $X$} as the intersection of all the $S$-subsemimodules of $M$ containing $X$. Conversely, given a subsemimodule $N$ of $M$, we will say that a subset $X$ of $M$ is a \emph{system of generators} for $N$~--- or that $X$ \emph{generates} $N$~--- if $N = \la X \ra$.

If $\{M_i\}$ is a family of $S$-semimodules, $M$ is an $S$-semimodule and $X$ is a non-empty set, the \emph{product} $\la \prod_{i \in I} M_i, +, (0_i)_{i \in I} \ra$ of the family $\{M_i\}_{i \in I}$, and $\la M^X, +, \0 \ra$ are clearly $S$-semimodules with the operations defined pointwise. $M^X$ is also called the \emph{power semimodule} of $M$ by $X$.

As in the case of ring modules, given a semiring $S$, an $S$-semimodule $M$ and a finite family $\{x_i\}_{i = 1}^n$ of elements of $M$, we call a \emph{linear combination} of the family $\{x_i\}$ any sum $\sum_{i = 1}^n a_i \cdot x_i$ with $a_i \in S$ for all $i =1, \ldots, n$.

It is obvious that, for any semiring $S$, if $M$ is an $S$-semimodule and $\varnothing \neq X \subseteq M$, then $\la X \ra = S \cdot X = \left\{\sum_{i = 1}^n a_i \cdot x_i \ \Big| \ a_i \in S, x_i \in X, n \in \N\right\}$, i.~e. $\la X \ra$ is the set of all the linear combinations of elements of $X$. Therefore, recalling that an $S$-semimodule is called \emph{cyclic} if it is generated by a single element $v$, such a semimodule shall be denoted also by $S \cdot v$.

We now investigate several basic constructions and properties of the categories of semimodules over semirings. For the general categorical definitions of the concepts involved (such as free and projective objects, products and coproducts etc.), we refer the reader to \cite{cats}. According to Remark~\ref{notation}, in all the definitions and statements regarding semimodules on a non-commutative semiring $S$, whenever we say simply $S$-semimodule or write $S\Mod$, we mean that the definition or the result holds for both left and right semimodules (suitably reformulated, where necessary).

Recalling that, if $S$ is a semiring and $X$ is a set, the \emph{support} of a map $f: X \lto S$ is the set $\supp f = \{x \in X \mid f(x) \neq 0_S\}$, we have
\begin{proposition}\label{freemod}
For any set $X$, the free $S$-semimodule $\fr_S(X)$ generated by $X$ is the set --- denoted by $S^{(X)}$ --- of functions from $X$ to $S$ with finite support, equipped with pointwise sum and scalar multiplication, and with the map $\chi: x \in X \lmapsto \chi_x \in S^{(X)}$, where $\chi_x$ is defined, for all $x \in X$, by
\begin{equation}\label{chi}
\chi_x(y) = \left\{\begin{array}{ll} 0_S & \textrm{if } y \neq x \\ 1 & \textrm{if } y = x\end{array}\right..
\end{equation}
\end{proposition}
\begin{proof}
Let $\la M, +, 0_M \ra$ be any $S$-semimodule and $f: X \lto M$ be an arbitrary map. We shall prove that there exists a unique $S$-semimodule morphism $h_f: S^{(X)} \lto M$ such that $h_f \circ \chi = f$. For any $\alpha \in S^{(X)}$, let us set $A = \supp\alpha$, and observe that $\alpha = \sum_{x \in A} \alpha(x) \cdot \chi_x$; then let us set 
\begin{equation}\label{hf}
h_f: \alpha  \in S^{(X)} \lmapsto \sum_{x \in A} \alpha(x) \cdot f(x) \in M.
\end{equation}
For any fixed $\ov x \in X$, $(h_f \circ \chi)(\ov x) = h_f(\chi_{\ov x}) = 1 \cdot f(\ov x) = f(\ov x)$, hence $h_f \circ \chi = f$.

The proof of the fact that $h_f$ is a semimodule homomorphism is straightforward. Moreover, if $h: S^{(X)} \lto M$ is an $S$-semimodule homomorphism such that $h \circ \chi = f$, for any $\alpha \in S^{(X)}$ we have $h(\alpha) = h\left(\sum_{x \in A} \alpha(x) \cdot \chi_x\right) = \sum_{x \in A} \alpha(x) \cdot h\left(\chi_x\right) = \sum_{x \in A} \alpha(x) \cdot (h\circ\chi)(x) = \sum_{x \in A} \alpha(x) \cdot f(x) = h_f(\alpha)$, hence $h_f$ is unique.
\end{proof}

Obviously, every $S$-semimodule is homomorphic image of a free one. Indeed, if $M$ is an $S$-semimodule, we can consider the free $S$-semimodule $S^{(M)}$ and the $S$-semimodule homomorphism $h_{\id_M}: S^{(M)} \lto M$ defined as in (\ref{hf}) by replacing $f$ with $\id_M$. It is immediate to verify that $h_{\id_M}$ is onto; moreover, by Proposition \ref{freemod}, $h_{\id_M} \circ \chi = \id_M$.

\begin{definition}\label{homsm}
Given $S$-semimodules $M$ and $N$, we define, on $\hom_S(M,N)$, the following operations and constants:
\begin{enumerate}
\item[-]for all $f, g \in \hom_S(M,N)$, the homomorphism $f + g$ is defined by $(f + g)(x) = f(x) + g(x)$, for all $x \in M$,
\item[-]$\0$ is the $0$-constant homomorphism,
\end{enumerate}
and, if $S$ is commutative,
\begin{enumerate}
\item[-]for all $a \in S$ and $f \in \hom_S(M,N)$, $a \cdot f$ is the map defined by $(a \cdot f)(x) = a \cdot f(x) = f(a \cdot x)$, for all $x \in M$.
\end{enumerate}
It is easy to see that $\la \hom_S(M,N), +, \0 \ra$ is a commutative monoid and, if $S$ is a commutative semiring, it is an $S$-semimodule with scalar multiplication $\cdot$. If $N = M$, the monoid (or, in case, the semimodule) of the endomorphisms $\hom_S(M,M)$ will be denoted by $\End_S(M)$.
\end{definition}



In order to show next result, we introduce the following notation. Given a semiring $S$ and two non-empty sets $X$ and $Y$, let $S^{X \times (Y)}$ be the commutative monoid of functions from $X \times Y$ to $S$ --- equipped with pointwise sum --- with finite support in the second variable, namely,
$$S^{X \times (Y)} \bydef \{k \in S^{X \times Y} \mid \forall x \in X \quad k(x,{}_-) \in S^{(Y)}\}.$$

It is easy to see that $S^{X \times (Y)}$ also enjoys a structure of $S$-bisemimodule in an obvious way.

\begin{theorem}\label{homofree}
Let $S$ be a semiring and $S^{(X)}$ and $S^{(Y)}$ free $S$-semimodules. The two commutative monoids $\hom_S(S^{(X)},S^{(Y)})$ and $S^{X \times (Y)}$ are isomorphic and, if $S$ is commutative, they are isomorphic as $S$-semimodules.
\end{theorem}
\begin{proof}
For any $k \in S^{X \times (Y)}$, let $h_k: S^{(X)} \lto S^{(Y)}$ be the map defined by $f \lmapsto \sum_{x \in X} f(x) k(x,-)$. Since $f$ has finite support it is easy to check that $h_k$ is a well-defined semimodule homomorphism.

Conversely, let us observe that, for any $f \in S^{(X)}$, $f = \sum_{x \in X} f(x) \chi_x$, with the maps $\chi_x$ defined by (\ref{chi}). Then, for all $h \in \hom_S(S^{(X)},S^{(Y)})$,
$$h(f) = h\left(\sum_{x \in X} f(x) \chi_x\right) = \sum_{x \in X} f(x) h(\chi_x).$$
Let $h \in \hom_S(S^{(X)},S^{(Y)})$ and define $k_h: (x,y) \in X \times Y \lmapsto h(\chi_x)(y) \in S$. It is easy to see that $k_h$ has finite support in the second variable --- that is $k \in S^{X \times (Y)}$ --- and that $h = h_{k_h}$.

So let $\eta: k \in S^{X \times (Y)} \lmapsto h_k \in \hom_S(S^{(X)},S^{(Y)})$; we shall now prove that $\eta$ is bijective. The fact that $\eta$ is surjective has just been proved. If $k \neq l \in S^{X \times (Y)}$, there exists a pair $(\ov x, \ov y) \in X \times Y$ such that $k(\ov x, \ov y) \neq l(\ov x, \ov y)$; then we have
$$h_k(\chi_{\ov x})(\ov y) = k(\ov x, \ov y) \neq l(\ov x, \ov y) = h_l(\chi_{\ov x})(\ov y),$$
whence $\eta$ is injective.

The fact that $\eta$ is a monoid homomorphism and, if $S$ is commutative, also a semimodule homomorphism is trivial.
\end{proof}

Notice that Theorem~\ref{homofree} yields a matrix representation of homomorphisms between finitely generated free semimodules: $\hom_S(S^m,S^n) \cong S^{m \times n}$ (in $\M$ or, if $S$ is commutative, in $S\Mod$), for all $m,n \in \N$.

\begin{theorem}\label{homo}
Let $M$ and $N$ be $S$-semimodules, $X$ and $Y$ be two sets of generators for $M$ and $N$ respectively, and $\pi: S^{(X)} \lto M$ and $\pi': S^{(Y)} \lto N$ be the canonical quotient morphisms.

Then, for any homomorphism $h: M \lto N$ there exists $k \in S^{X \times (Y)}$ such that $h \circ \pi = \pi' \circ h_k$.
\end{theorem}
\begin{proof}
Consider the following diagram
\begin{equation*}
\xymatrix{
S^{(X)} \ar@{-->}[rr]^{\ov h} \ar[dd]_\pi && S^{(Y)} \ar[dd]^{\pi'} \\
&&\\
M	\ar[rr]_h && N
}.
\end{equation*}
First observe that in any construct (namely, a category which is concrete over the one of sets) free objects are projective; then the existence of the morphism $\ov h$ closing such a diagram follows immediately from the fact that $S^{(X)}$ is free. Moreover we know by Theorem~\ref{homofree} that $\ov h = h_k$ for some $k \in S^{X \times (Y)}$.

Nonetheless it is interesting to notice that, since each element $m$ of $M$ can be written as $\sum_{x \in X} a_x \pi(\chi_x)$, with the $a_x$'s in $S$, then
$$\begin{array}{lllll}
h(m) & = & h\left(\sum_{x \in X} a_x \pi(\chi_x)\right) & = & \sum_{x \in X} a_x h(\pi(\chi_x)) \\
		 & = & \sum_{x \in X} a_x \pi'(h_k(\chi_x)) & = & \sum_{x \in X} a_x \pi'(k(x,{}_-)).
\end{array}$$
This shows how ``concretely'' $k$ determines $h$.
\end{proof}

\section{Projective semimodules}
\label{proj}

In this section we will show some results about projective semimodules and, in particular, we shall characterize finitely generated projective semimodules. Such a characterization will play a prominent role in the proof of the functorial nature of the construction of the \gr group of a semiring.

\begin{proposition}\cite[Proposition (17.16)]{golan}\label{projgol}
For any semiring $S$ an $S$-semimodule $M$ is projective if and only if it is a retract of a free $S$-semimodule.
\end{proposition}


Let $S$ be a semiring and, for all $n \in \N$, let $M_{n}(S)$ be the set of all $n \times n$ square matrices of elements of $S$. It is easy to verify that the structure $\la M_{n}(S), +, \star, o, \iota \ra$, where
\begin{itemize}
\item $o$ is the $0_S$-constant matrix,
\item $\iota$ is the matrix whose components are defined by $\id_{ij} = \left\{\begin{array}{ll} 1 & \textrm{if } i=j \\ 0_S & \textrm{otherwise} \end{array}\right.$,
\item $+$ is the componentwise sum,
\item the operation $\star$ is defined by $(a_{ij}) \star (b_{ij}) = \left(\sum_{k=1}^n a_{ik} b_{kj}\right)$,
\end{itemize}
is a semiring, called the \emph{semiring of $n \times n$ square matrices} over $S$; moreover, such a semiring is isomorphic to a familiar one, as we are going to show.

\begin{theorem}\label{matrixendo}
The semirings $M_n(S)$ and $\End_S(S^n)$ are isomorphic, for any semiring $S$ and any natural number $n$.
\end{theorem}
\begin{proof}
By Theorem~\ref{homofree}, there exists an isomorphism $\eta: a \in M_n(S) \lmapsto h_a \in \End_S(S^n)$ between the additive monoid reducts of the two semirings; we shall prove that $\eta$ is actually a semiring isomorphism, with the product in $\End_S(S^n)$ defined as the composition in the reverse order.

First, it can be immediately observed that $\eta(\iota) = \id$. Now, given a matrix $a = (a_{ij})$, for all $\vec v = (v_1, \ldots, v_n) \in S^n$, by definition $h_a(\vec v)$ is the matrix product $\vec v \star a = \left(\sum_{i=1}^n v_i a_{ik}\right)$. Hence it follows immediately that $h_{a \star b} (\vec v) = \vec v \star (a \star b) = (\vec v \star a) \star b = (h_b \circ h_a)(\vec v) = (h_a h_b)(\vec v)$, for all $a, b \in M_n(S)$ and $\vec v \in S^n$. Therefore $\eta$ is a semiring isomorphism.
\end{proof}

\begin{theorem}\label{finproj}
An $n$-generated $S$-semimodule $M$ is projective if and only if there exist $\vec u_1, \ldots, \vec u_n \in S^n$ such that $M \cong S \cdot \{\vec u_i\}_{i=1}^n$ and the matrix $(u_{ij})$ is a multiplicatively idempotent element of the semiring $M_{n}(S)$.
\end{theorem}
\begin{proof}
Let $M$ be a projective $S$-semimodule, $\{v_i\}_{i=1}^n$ a generating set for $M$ and, for all $i = 1, \ldots, n$, let $\vec e_i$ be the $n$-tuple of elements of $S$ whose entries are all equal to $0_S$ except the $i$-th which is equal to $1$. Since $S^n$ is free over $\{\vec e_i\}_{i=1}^n$, the map sending each $\vec e_i$ to $v_i$ can be extended to a unique homomorphism $\pi: S^n \lto M$ which is obviously onto. Then the projectivity of $M$ implies the existence of a morphism $\mu: M \lto S^n$ such that $\pi \circ \mu = \id_M$, and $\mu$ is injective. So, if we set $\vec u_i = \mu(v_i)$, for all $i = 1, \ldots, n$, we have $M \cong S \cdot \{\vec u_i\}_{i=1}^n$ and we can identify the two semimodules and $\pi$ with the morphism extending the map which sends $\vec e_i$ to $\vec u_i$ for each $i$.

Now we have that $\pi(\vec v) = \vec v$ for all $\vec v \in S \cdot \{\vec u_i\}_{i=1}^n$, hence in particular
$$\begin{array}{l}
(u_{i1}, \ldots, u_{in}) = \vec u_i = \pi(\vec u_i) = \pi(u_{i1}, \ldots, u_{in}) \\ = \pi\left(\sum\limits_{k = 1}^n u_{ik} \cdot \vec e_k\right) = \sum\limits_{k=1}^n u_{ik} \cdot \pi(\vec e_k) = \sum\limits_{k=1}^n u_{ik} \cdot \vec u_k, \\
\textrm{for all } i = 1, \ldots, n,
\end{array}$$
that is, for all $i,j = 1, \ldots, n$, $u_{ij} = \sum\limits_{k=1}^n u_{ik} u_{kj}$. Therefore $(u_{ij}) \star (u_{ij}) = (u_{ij})$, i.~e. $(u_{ij})$ is a multiplicatively idempotent element of $M_n(S)$.

The converse implication is obtained by proving that $S \cdot \{\vec u_i\}_{i=1}^n$ is a retract of the free $S$-semimodule $S^n$, and therefore it is projective by Proposition \ref{projgol}. Indeed, if $(u_{ij})$ is a multiplicatively idempotent element of $M_n(S)$, for any element $\vec v = \sum_{i=1}^n a_i \cdot \vec u_i \in S \cdot \{\vec u_i\}_{i=1}^n$,
$$\begin{array}{l}
\pi(\vec v) = \pi\left(\sum\limits_{i=1}^n a_i \cdot \vec u_i\right) = \pi\left(\sum\limits_{i=1}^n a_i \cdot \sum\limits_{k=1}^n u_{ik} \cdot \vec e_k\right) \\
= \sum\limits_{i=1}^n a_i \cdot \sum\limits_{k=1}^n u_{ik} \cdot \vec u_k = \sum\limits_{i=1}^n a_i \cdot \vec u_i = \vec v,
\end{array}$$
whence $S \cdot \{\vec u_i\}_{i=1}^n$ is projective.
\end{proof}

\begin{remark}\label{cyclic}
A cyclic $S$-semimodule $M$ is projective if and only if there exists $u \in S$ such that $M \cong S \cdot u$ and $u^2 = u$.
\end{remark}


\section{MV-semirings and MV-semimodules}
\label{mvsemirings}

In~\cite{bdn} and~\cite{dng} semirings were studied in connection with MV-algebras --- the algebraic semantics of \L ukasiewicz infinite-valued propositional logic. In this section we recall the definition of MV-algebra and, briefly, some of the contents of the aforementioned papers; for a comprehensive study of MV-algebras we refer the reader to \cite{mvbook}. 

\begin{definition}\label{mvalg}
An \emph{MV-algebra} is an algebra $\la A, \oplus, ^*, \z \ra$ of type $(2,1,0)$ such that $\la A, \oplus, \z \ra$ is a commutative monoid, and, for all $x,y \in A$,
\begin{enumerate}[(MV1)]
\item $(x^*)^* = x$;
\item $x \oplus \z^* = \z^*$;
\item $(x^* \oplus y)^* \oplus y = (y^* \oplus x)^* \oplus x$.
\end{enumerate}
\end{definition}

Since Definition \ref{mvalg} can be formulated in the language of Universal Algebra by means of equations, MV-algebras form a variety. Congruences and homomorphisms are defined in an obvious way, namely, as equivalence relations that are compatible with $\oplus$ and $^*$ and functions that preserve the operations and the constant $\z$ respectively. 

On every MV-algebra $A$ it is possible to define another constant $\u = \z^*$ and the operation $\odot$ by $x \odot y = (x^* \oplus y^*)^*$; moreover, for all $x, y \in A$, the following well-known properties hold:
\begin{enumerate}
\item[-] $\la A, \odot, ^*, \u\ra$ is an MV-algebra;
\item[-] $^*$ is an isomorphism between $\la A, \odot, ^*, \u\ra$ and $\la A, \oplus, ^*, \z\ra$;
\item[-] $\u^* = \z$;
\item[-] $x \oplus y = (x^* \odot y^*)^*$;
\item[-] $x \oplus \u = \u$ \quad (reformulation of (MV2));
\item[-] $x \oplus x^* = \u$.
\end{enumerate}

\begin{lemma}\cite{mvbook}\label{order}
Let $A$ be an MV-algebra and $x, y \in A$. The following conditions are equivalent:
\begin{enumerate}[(a)]
\item $x^* \oplus y = \u$;
\item $x \odot y^* = \z$;
\item $y = x \oplus (y \odot x^*)$;
\item there exists an element $z \in A$ such that $x \oplus z = y$.
\end{enumerate}
\end{lemma}

For any MV-algebra $A$ and $x, y \in A$, we write $x \leq y$ if and only if $x$ and $y$ satisfy the equivalent conditions of Lemma \ref{order}. It is well-known that $\leq$ is a partial order on $A$, called the \emph{natural order} of $A$. Moreover, the natural order determines a structure of bounded distributive lattice on $A$ \cite[Propositions 1.1.5 and 1.5.1]{mvbook}, with $\z$ and $\u$ respectively bottom and top element, and $\vee$ and $\wedge$ defined by
\begin{eqnarray*}
&& x \vee y = (x \odot y^*) \oplus y, \\ 
&& x \wedge y = (x^* \vee y^*)^* = x \odot (x^* \oplus y). 
\end{eqnarray*}

A subset $I$ of an MV-algebra $A$ is called an \emph{ideal} if it is a downward closed submonoid of $\la A, \oplus, \z\ra$, i.~e. if it satisfies the following properties:
\begin{itemize}
\item $\z \in I$;
\item $I$ is downward closed, that is, $b \leq a$ implies $b \in I$ for all $a \in I$ and $b \in A$;
\item $a \oplus b \in I$ for all $a, b \in I$.
\end{itemize}

The following proposition shows that MV-algebraic congruences are in one-one correspondence with MV-ideals.

\begin{proposition}\emph{\cite[Proposition 1.2.6]{mvbook}}\label{ideal}
Let $I$ be an ideal of an MV-algebra $A$. Then the binary relation $\sim_I$ defined by ``$a \sim_I b$ iff $d(a,b) \bydef (a \odot b^*) \oplus (b \odot a^*) \in I$'' is a congruence on $A$, and $[\z]_{\sim_I} = I$.

Conversely, if $\sim$ is a congruence on $A$, then $[\z]_\sim$ is an ideal and $a \sim b$ iff $d(a,b)  \in [\z]_\sim$.
\end{proposition}

In the light of Proposition \ref{ideal}, for any MV-algebra $A$ and any ideal $I$ of $A$, we shall denote by $I$ also the congruence $\sim_I$, by $A/I$ the corresponding quotient MV-algebra and by $a/I$ the congruence class of any given element $a \in A$.

\begin{exm}
Consider the interval $[0,1]$ of $\R$ with the operations $\oplus$ and $^*$ defined, respectively, by $x \oplus y \bydef \min\{x + y,1\}$ and $x^* \bydef 1-x$. Then structure $\la [0,1], \oplus, ^*, 0\ra$ is an MV-algebra, often called the \emph{standard MV-algebra}. The reason for such a name is the fact (which is perfectly equivalent to Theorem \ref{stand} below) that the algebra $[0,1]$ generates the whole variety of MV-algebras, namely, every MV-algebra can be obtained as a quotient of a subalgebra of a Cartesian power $[0,1]^\kappa$ (with pointwise defined operations) for some cardinal $\kappa$.

In the standard MV-algebra the order relation (and therefore the lattice structure) is the usual one of real numbers; the product $\odot$ is defined by $x \odot y \bydef \max\{0, x+y-1\}$.
\end{exm}

\begin{theorem}\emph{\cite[Theorem 2.5.3]{mvbook}}\label{stand}
An equation holds in $[0,1]$ if and only if it holds in every MV-algebra.
\end{theorem}

\begin{exm}\label{0u}
Let $\la G, +, -, 0, \vee, \wedge \ra$ be a lattice-ordered Abelian group, let $u$ be a fixed positive element of $G$ and $[0,u] = \{x \in G \mid 0 \leq x \leq u\}$. Now let us define, for all $x, y \in [0,u]$, $x \oplus y \bydef (x + y) \wedge u$ and $x^* \bydef u - x$. Then it is easy to check that the structure $\la [0,u], \oplus, ^*, 0\ra$ is an MV-algebra.

So, according to Example \ref{l-group}, given an idempotent semifield $\la F, \wedge, +, -, \top, 0\ra$, for any fixed $u \in F$, with $0< u < \top$, it is possible to build the MV-algebra $\la [0,u], \oplus, ^*, 0\ra$ as above.
\end{exm}

\begin{exm}
For any Boolean algebra $\la B, \vee, \wedge, ', 0, 1\ra$, the structure $\la B, \vee, ', 0\ra$ is an MV-algebra. Boolean algebras form a subvariety of the variety of MV-algebras. They are precisely the MV-algebras satisfying the additional equation $x \oplus x = x$.
\end{exm}

For the proof of the following result we refer the reader to \cite[Proposition 3.6]{dng}.
\begin{proposition}\label{mvsemi}
Let $A$ be an MV-algebra. Then $A\vr = \la A, \vee, \odot, \z, \u \ra$ and $A\wr = \la A, \wedge, \oplus, \u, \z \ra$ are semirings. Moreover, the involution $^*: A \lto A$ is an isomorphism between them.
\end{proposition}

\begin{remark}\label{star}
Thanks to Proposition~\ref{mvsemi}, we can limit our attention to one of the two \emph{semiring reducts of $A$}; therefore, whenever not differently specified, we will refer only to $A\vr$, all the results holding also for $A\wr$ up to the application of $^*$.
\end{remark}

We recall the following definition from \cite{bdn}.
\begin{definition}\label{luksemi}
An \emph{MV-semiring} is a commutative, additively idempotent semiring $\la A, \vee, \cdot, 0, 1 \ra$ for which there exists a map $^*: A \lto A$ --- called the \emph{negation} --- satisfying, for all $a, b \in A$, the following conditions:
\begin{enumerate}[(i)]
\item $a \cdot b = 0$ iff $b \leq a^*$ (where $a \leq b$ iff $a \vee b = b$);
\item $a \vee b = (a^* \cdot (a^* \cdot b)^*)^*$.
\end{enumerate}
\end{definition}

The following result is basically a reformulation of \cite[Proposition 2.2]{bdn}.
\begin{proposition}\label{luksemiprop}
For any MV-algebra $\la A, \oplus, \z\ra$, both the semiring reducts $A\vr$ and $A\wr$ are MV-semirings. Conversely, if $\la A, \vee, \cdot, 0, 1\ra$ is an MV-semiring with negation $^*$, the structure $\la A, \oplus, ^*, 0 \ra$, with
\begin{equation*}
a \oplus b = (a^* \cdot b^*)^* \quad \textrm{for all $a, b \in S$},
\end{equation*}
is an MV-algebra.
\end{proposition}


We now start considering semimodules over MV-semirings (\emph{MV-semimodules} for short). Recalling that a semilattice is a commutative idempotent semigroup, we observe that, if $\la M, +, 0\ra$ is a semimodule over an additively idempotent semiring (hence, in particular, over an MV-semiring), then it is necessarily a semilattice with neutral element $0$. Indeed the sum, besides being commutative by definition, is also idempotent since $x = 1x = (1 \vee 1) x = 1x + 1x = x + x$. In what follows, for an MV-algebra $A$, we will use the notation of join-semilattices $\la M, \vee, \bot \ra$ for the $A\vr$-semimodules and the one of meet-semilattices $\la M, \wedge, \top\ra$ for the $A\wr$-semimodules. We also recall that we denote by $\sL$ the category whose objects are idempotent monoids or, that is the same, semilattices with identity, and morphisms are monoid homomorphisms; whenever we refer to an order relation in such structures it is understood that such an order is the one naturally induced by the semilattice operation.

We shall now enforce the characterization of cyclic projective semimodules in the case of MV-algebras. For the role that the idempotent elements of MV-algebras play in this context, we briefly recall some basic facts about them \cite[Section 1.5]{mvbook}, most of which are very easy to check.
\begin{itemize}
\item An element $a$ of an MV-algebra $A$ is called \emph{idempotent} or \emph{Boolean} if $a \oplus a = a$.
\item For any $a \in A$, $a \oplus a = a$ iff $a \odot a = a$.
\item An element $a$ is Boolean iff $a^*$ is Boolean.
\item If $a$ and $b$ are idempotent, then $a \oplus b$ and $a \odot b$ are idempotent as well; moreover we have $a \oplus b = a \vee b$, $a \odot b = a \wedge b$, $a \vee a^* = \u$ and $a \wedge a^* = \z$.
\item The set $\B(A) = \{a \in A \mid a \oplus a = a\}$ is a Boolean algebra, usually called the \emph{Boolean center} of the MV-algebra $A$.
\item For any $a \in A$ and $u \in \B(A)$, $a = (a \oplus u) \wedge (a \oplus u^*) = (a \odot u) \vee (a \odot u^*)$. 
\end{itemize}

\begin{proposition}\label{cyclicmv}
Let $A$ be an MV-algebra and $M$ a cyclic $A$-semimodule. Then the following are equivalent:
\begin{enumerate}
\item[$(a)$] $M$ is projective,
\item[$(b)$] there exists $u \in A$ such that $M \cong A \odot u$ and $u \odot u = u$,
\item[$(c)$] $M$ is a direct summand of the free semimodule $A$, i.~e. there exists a subsemimodule $N$ of $A$ such that $A \cong M \oplus N$.
\end{enumerate}
\end{proposition}
\begin{proof}
The equivalence between $(a)$ and $(b)$ is already known (Remark~\ref{cyclic}) and holds for cyclic semimodules over any semiring.

Let us consider a cyclic projective semimodule $M = A \odot u$ with $u$ idempotent element of $A$ and let $N = A \odot u^*$. Then we have inclusion $A$-semimodule morphisms given by $i: M \lto A$ and $i^*: N \lto A$. Now given any $A$-semimodule $P$ with $A$-semimodule homomorphisms $f: M \lto P$ and $g: N \lto P$, we may define the map
$$\begin{array}{llll}
h: & A & \lto & P \\
 & a & \lmapsto & f(a \odot u) \vee g(a \odot u^*).
\end{array}$$
The fact that $h$ is an $A$-semimodule homomorphism is obvious. Now for any $b \in M$, we may write $b = a \odot u$ for some $a \in A$, so $h(i(b)) = h(b) = f(b \odot u) \vee g(b \odot u^*) = f(a \odot u \odot u) \vee g(a \odot u \odot u^*) = f(a \odot u) \vee g(\z) = f(b)$ and, analogously, $h(i^*(b)) = g(b)$ for any $b \in N$.

If $k: A \lto P$ is another morphism such that $k \circ i = f$ and $k \circ i^* = g$, for any $a \in A$, since $u$ is idempotent, $a$ can be decomposed as $(a \odot u) \vee (a \odot u^*)$, and we have:
$$k(a) = k((a \odot u) \vee (a \odot u^*)) = k(a \odot u) \vee k(a \odot u^*) = f(a \odot u) \vee g(a \odot u^*) = h(a).$$
So any sink with domain $\{M, N\}$ induces a unique morphism, with the same codomain and with domain $A$, such that the compositions of the inclusion maps with it coincide with the morphisms of the sink. In other words $A$ is the coproduct, i.~e. the direct sum, of $M$ and $N$.

Now assume that $(c)$ holds and consider the sink $\{\id_M: M \lto M, \bot: N \lto M\}$, $\bot$ being the $\bot$-constant morphism. Since $A \cong M \oplus N$, there exists a unique morphism $\pi: A \lto M$ extending such a sink. In particular $\pi \circ i = \id_M$, so $M$ is a retract of $A$, hence it is projective.
\end{proof}

\begin{remark}\label{0top}
It is worth to underline that, for the $\wedge$-$\oplus$ reduct $\la A, \wedge, \oplus, \u, \z\ra$ of an MV-algebra and for an idempotent semifield $\la F, \wedge, +, -, \top, 0\ra$, $\z$ and $0$ are the respective identities of the semiring product, the neutral elements for the semiring sum being, respectively, $\u$ and $\top$. So, for any set $X$ and functions $f: X \lto A$ and $g: X \lto F$, we have
\begin{eqnarray*}
\supp f & = & \{x \in X \mid f(x) \neq \u\}, \\ 
\supp g & = & \{x \in X \mid g(x) \neq \top\}. 
\end{eqnarray*}
Analogously, equation (\ref{chi}) becomes, respectively,
\begin{eqnarray}
\chi_x(y) & = & \left\{\begin{array}{ll} \u & \textrm{if } y \neq x \\ \z & \textrm{if } y = x\end{array}\right. \label{chimv}\\
\chi_x(y) & = & \left\{\begin{array}{ll} \top & \textrm{if } y \neq x \\ 0 & \textrm{if } y = x\end{array}\right.. \label{chisf}\\ \nonumber
\end{eqnarray}
\end{remark}

\section{Strong MV-semimodules}
\label{reprsec}

In this section we introduce the class of ``strong'' MV-semimodules and present some relevant examples. The defining property of such semimodules is basically a good behaviour of the scalar multiplication with respect to the MV-algebraic involution $^*$. Moreover, strong semimodules allow us to specialize Proposition~\ref{semiend} and Corollary \ref{semirepr} to MV-semirings as shown respectively in Proposition \ref{endmv} and Corollary \ref{mvrepr1}.

\begin{definition}\label{mvmod}
Let $A$ be an MV-semiring and $M$ an $A$-semimodule. $M$ is said to be a \emph{strong $A$-semimodule} provided it fulfils, for all $a, b \in A$, the following additional condition:
\begin{equation}\label{mvmodstar}
a \cdot x = b \cdot x \quad \textrm{for all } x \in M \qquad \textrm{implies} \qquad a^* \cdot x = b^* \cdot x \quad \textrm{for all } x \in M.
\end{equation}
\end{definition}

\begin{exm}\label{strong}
For any MV-algebra $A$, $\la A, \vee, \z \ra$ is a strong $A\vr$-semimodule as well as $\la A, \wedge, \u  \ra$ is a strong $A\wr$-semimodule. It is easy to see also that any free MV-semimodule is strong.
\end{exm}
\begin{exm}\label{nonstrong}
Let $A$ be the MV-algebra $[0,1]$ and consider the join-semilattice $M = \left\la\left[0,\frac{1}{2}\right], \vee, 0 \right\ra$ as an $A$-semimodule with $\odot$ as the scalar multiplication. For any $a \leq 1/2$, $a \odot x = 0 \odot x$ for all $x \in M$ but, if we set for example $a = x = 1/2$, $a^* \odot x = 0$ while $0^* \odot x = 1/2$. Hence $M$ is not a strong MV-semimodule.
\end{exm}

\begin{proposition}\label{endmv}
Let $A$ be an MV-semiring and $M$ a semilattice with neutral element. Then $M$ is a strong $A$-semimodule if and only if $\End_{\sL}(M)$ --- which in general is not an MV-semiring --- contains an MV-subsemiring that is homomorphic image of $A$ (in $\MV$).
\end{proposition}
\begin{proof}
By Proposition~\ref{semiend}, $M$ is an $A$-semimodule if and only if there exists a semiring homomorphism $\xi: A \lto \End_{\sL}(M)$.

Now, assume that $M$ is a strong $A$-semimodule and let us consider the semiring homomorphism $\xi: A \lto \End_{\sL}(M)$ defined in (\ref{xi}). Then $\xi[A]$ is a commutative subsemiring of $\End_{\sL}(M)$. The map $^*: \xi(a) \in \xi[A] \lmapsto \xi(a^*) \in \xi[A]$ is well-defined; indeed, if $\xi(a) = \xi(b)$, then $\xi(a^*) = \xi(b^*)$ by (\ref{mvmodstar}). Condition (ii) of Definition~\ref{luksemi} follows easily from the definition of $^*$ on $\xi[A]$ and from the fact that $\xi$ is a homomorphism. On the other hand, if $\xi(ab) = 0$ then, by (ii) of Definition~\ref{luksemi},
$$\xi(a)^* \vee \xi(b) = (\xi(a) \xi(ab)^*)^* = (\xi(a) \id_M)^* = \xi(a)^*,$$
i.~e. $\xi(a)\xi(b) = 0$ implies $\xi(b) \leq \xi(a)^*$, the other implication being obvious.

Conversely, if $\xi[A]$ is an MV-semiring and $\xi$ preserves the $^*$, then condition (\ref{mvmodstar}) is trivially verified for the action $\cdot: (a,x) \in A \times M \lmapsto (\xi(a))(x) \in M$ hence $M$ is a strong $A$-semimodule.
\end{proof}

The following result is an immediate consequence of Proposition \ref{endmv}, Example \ref{strong} and Corollary \ref{semirepr}.
\begin{corollary}\label{mvrepr1}
Any MV-semiring $A$ can be embedded in the endomorphism semiring $\End_{\sL}\left({A^\vee}^{(X)}\right)$ of a join-semilattice of functions with finite support from a non-empty set $X$ to $A$, with pointwise join.

In particular $A$ is embeddable in the semiring $\End_{\sL}(A^\vee)$ of the endomorphisms of its join-semilattice reduct.
\end{corollary}



%

The proof of the following proposition is trivial.
\begin{proposition}\label{}
The following hold for any MV-algebra $A$.
\begin{enumerate}[(i)]
\item For any MV-ideal $I$, the semilattice reduct of the quotient MV-algebra $A/I$ is a strong semimodule over $A$ with
$$\begin{array}{cccc}
\cdot: & A \times A/I & \lto & A/I \\
			 & (a, x/I) & \lmapsto & (a \odot x)/I.
\end{array}$$
\item If $B$ is an MV-algebra and $h \in \hom_{\MV}(A, B)$, every strong $B$-semimodule $N$ is a strong $A$-semimodule with
$$\begin{array}{cccc}
\cdot_A: & A \times N & \lto & N \\
			 & (a, x) & \lmapsto & h(a) \cdot_B x.
\end{array}$$
\end{enumerate}
\end{proposition}

\section{The Grothendieck group of an MV-algebra}
\label{gro}

In this section we shall construct, for any MV-algebra $A$, its Grothendieck group $K_0A$, and we will prove that such a construction defines a functor between the categories $\cat{MV}$ of MV-algebras and $\Gab$ of Abelian groups. 


\begin{definition}\label{k0}
Let $S$ be a semiring, $\la\P(S), \oplus, [\{0\}]\ra$ the Abelian monoid of isomorphism classes of finitely generated projective left $S$-semimodules and let $J = \frab(\P(S))$ the free Abelian group generated by such isomorphism classes. For any finitely generated projective left $S$-semimodule $P$, we denote by $[P]$ its isomorphism class. Let $H$ be the subgroup of $J$ generated by all the expressions $[P] + [Q] - [P \oplus Q]$.

We define the \emph{\gr group} of a semiring $S$ to be the factor group $J/H$ and denote this by $K_0S$. Since the two semiring reducts of an MV-algebra $A$ are isomorphic, we may define the \gr group $K_0A$ up to isomorphism as the \gr group of either of its semiring reducts.
\end{definition}

\begin{lemma}\label{k0mfree}
For any semiring $S$, if we consider $\Gab$ as a concrete category over the one --- $\cat M^{\text{Ab}}$ --- of Abelian monoids, $K_0S$ is $\cat M^{\text{Ab}}$-free over $\la \P(S), \oplus, [\{0\}]\ra$, with associated monoid morphism
\begin{equation}\label{k}
k_S: [P] \in \P(S) \lmapsto [P]/H \in K_0S.
\end{equation}
\end{lemma}
\begin{proof}
We need to prove the following universal property: for any Abelian group $W$ and any monoid homomorphism $f: \P(S) \lto W$, there exists a unique group homomorphism $g: K_0S \lto W$ such that $g \circ k_S = f$.

So let $W$ be an Abelian group, $f: \P(S) \lto W$ a monoid homomorphism, and $i: \P(S) \lto \frab(\P(S))$ the inclusion map. Then there exists a unique group homomorphism $f': \frab(\P(S)) \lto W$ such that $f' \circ i = f$. On the other hand, since $f$ is a monoid homomorphism, the kernel of $f'$ obviously contains $H$ and, therefore, $f'$ induces a unique group homomorphism $g: K_0S \lto W$ such that $g \circ \pi = f'$, where $\pi: \frab(\P(S)) \lto K_0S$ is the canonical projection.

So we have $g \circ \pi \circ i = f' \circ i = f$ but, clearly, $\pi \circ i = k_S$, hence $g \circ k_S = f$ and $g$ is unique with such a property, i.~e., $K_0S$ is $\cat M^{\text{Ab}}$-free over $\P(S)$.
\end{proof}


\begin{lemma}\label{k0lemma}
Let $A$ and $B$ be two MV-algebras. Any MV-homomorphism $f: A \lto B$ induces a monoid homomorphism from $\P(A)$ to $\P(B)$.
\end{lemma}
\begin{proof}
By Theorem~\ref{finproj}, finitely generated projective semimodules over a semiring can be identified with multiplicatively idempotent square matrices with values in the same semiring. It is immediate to verify that, if $M \cong A \cdot (u_{ij})_{i,j=1}^m$ and $N \cong A \cdot (v_{ij})_{i,j=1}^n$ are finitely generated projective $A$-semimodules, the finitely generated projective $A$-semimodule $M \oplus N$ is isomorphic to $A \cdot (w_{ij})_{i,j=1}^{m+n}$ with
\begin{equation}\label{matrsum}
w_{ij} = \left\{\begin{array}{ll}
u_{ij} & \text{if } i,j \leq m \\
v_{(i-m)(j-m)} & \text{if } i,j > m \\
0 & \text{otherwise}
\end{array}\right..
\end{equation}
Moreover, if $M$ and $N$ are isomorphic, we can assume the corresponding matrices to have the same size. Indeed, suppose $m < n$, $M \cong M \oplus \underbrace{\{0\} \oplus \cdots \oplus \{0\}}_{n-m \text{ times}}$, hence the $m \times m$ matrix $(u_{ij})$ generates a semimodule isomorphic to the one generated by the $n \times n$ matrix $(u'_{ij})$ which coincides with $(u_{ij})$ on every entry $ij$ such that $i,j \leq m$ and is constantly equal to zero elsewhere.

Let now $(u_{ij})$ be an idempotent $n \times n$ $A$-matrix and consider the $B$-matrix $(f(u_{ij}))$. Since $f$ is an MV-homomorphism, it preserves all the MV-algebraic operations and the lattice structure, hence it is also a semiring homomorphism. So we have
$$\left(\bigvee_{k=1}^n f(u_{ik}) \odot f(u_{kj})\right) = \left(f\left(\bigvee_{k=1}^n u_{ik} \odot u_{kj}\right)\right) = (f(u_{ij})),$$
whence $(f(u_{ij})) \star (f(u_{ij})) = (f(u_{ij}))$ and $(f(u_{ij}))$ is an idempotent $n \times n$ $B$-matrix.

Now assume that $M \cong A \cdot (u_{ij})_{i,j=1}^n$ and $N \cong A \cdot (v_{ij})_{i,j=1}^n$ are isomorphic finitely generated projective $A$-semimodules. Then, for all $i = 1, \ldots, n$, there exist $a_{i1}, \ldots, a_{in}, b_{i1}, \ldots, b_{in} \in A$ such that $(u_{i1}, \ldots, u_{in}) = \bigvee_{k=1}^n a_{ik} \cdot (v_{k1}, \ldots, v_{kn})$ and $(v_{i1}, \ldots, v_{in}) = \bigvee_{k=1}^n b_{ik} \cdot (u_{k1}, \ldots, u_{kn})$. So each vector $f(\vec u_i) = (f(u_{i1}), \ldots, f(u_{in}))$ can be expressed as a linear combination of the vectors $\{f(\vec v_i)\}_{i=1}^n$ with the scalars $f(a_{i1}), \ldots, f(a_{in}) \in B$ and, conversely, each $f(\vec v_i)$ can be expressed as a linear combination of the vectors $\{f(\vec u_i)\}_{i=1}^n$ with the scalars $f(b_{i1}), \ldots, f(b_{in}) \in B$; this means that the subsemimodules of $B^n$ generated respectively by $\{f(\vec u_i)\}_{i=1}^n$ and $\{f(\vec v_i)\}_{i=1}^n$ are isomorphic.

The above guarantees that
\begin{equation}\label{fchap}
\begin{array}{cccc}
\hat f: & \P(A) & \lto & \P(B) \\
		& \left[A \cdot (u_{ij})\right] & \lmapsto & \left[B \cdot (f(u_{ij}))\right]
\end{array}
\end{equation}
is a well-defined map. The fact that $\hat f\left(\left[\{0\}\right]\right) = \left[\{0\}\right]$ is obvious. On the other hand, given two classes $\left[A \cdot (u_{ij})_{i,j=1}^m\right], \left[A \cdot (v_{ij})_{i,j=1}^n\right] \in \P(A)$, the semimodule $A \cdot (u_{ij}) \oplus A \cdot (v_{ij})$ is isomorphic to $A \cdot (w_{ij})_{i,j=1}^{m+n}$ with $(w_{ij})$ defined by (\ref{matrsum}), and
$$f(w_{ij}) = \left\{\begin{array}{ll}
f(u_{ij}) & \text{if } i,j \leq m \\
f(v_{(i-m)(j-m)}) & \text{if } i,j > m \\
f(0) = 0 & \text{otherwise}
\end{array}\right.,$$
whence
$$\hat f\left(\left[A \cdot (u_{ij}) \oplus A \cdot (v_{ij})\right]\right) = [B \cdot (f(u_{ij})) \oplus B \cdot (f(v_{ij}))] = \hat f\left(\left[A \cdot (u_{ij})\right]\right) \oplus \hat f\left(\left[A \cdot (v_{ij})\right]\right),$$
and the assertion is proved.
\end{proof}

Thanks to Lemma~\ref{k0lemma}, we can now prove the following

\begin{theorem}\label{k0thm}
$K_0$ is a functor from $\MV$ to $\Gab$.
\end{theorem}
\begin{proof}
Let $A$ and $B$ be two MV-algebras, $f$ a homomorphism between them, and $k_B: \P(B) \lto K_0B$ the monoid homomorphism defined by (\ref{k}). By Lemma \ref{k0lemma}, we may define the monoid homomorphism $\hat f: \P(A) \lto \P(B)$ as in (\ref{fchap}). Then $k_B \circ \hat f: \P(A) \lto K_0B$ is a monoid homomorphism and, by Lemma~\ref{k0mfree}, it can be extended in a unique way to a group homomorphism $K_0 f: K_0A \lto K_0B$.

The fact that $K_0$ preserves identity morphisms and morphism composition is a trivial consequence of (\ref{fchap}). The theorem is proved.
\end{proof}

As immediate generalizations of Lemma~\ref{k0lemma} and Theorem~\ref{k0thm}, we have the following results.

\begin{lemma}\label{k0lemmasr}
Let $S$ and $T$ be two semirings. Any semiring homomorphism $f: S \lto T$ induces a monoid homomorphism from $\P(S)$ to $\P(T)$.
\end{lemma}

\begin{theorem}\label{k0thmsr}
$K_0$ is a functor from the category $\sR$ of (unital) semirings to $\Gab$.
\end{theorem}
\begin{proof}
It is immediate to verify that Lemma~\ref{k0lemma} can be stated and proved for semirings in a completely analogous way. Then the argument of the proof of Theorem~\ref{k0thm} can be applied straightforwardly.
\end{proof}

\section{MV-semimodules and semimodules over idempotent semifields with strong unit}
\label{functor}

For a lattice-ordered Abelian group $G$, an element $u > 0$ is called a \emph{strong order unit} if, for all $x \in G$, $x > 0$, there exists $n \in \N$ such that $n u > x$. In the category $\ell\Gab_u$ of Abelian $\ell$-groups with a distinguished strong order unit the morphisms are $\ell$-group homomorphisms which send the distinguished strong unit of the domain to the one of the codomain.

In~\cite{mun} a categorical equivalence between the category $\ell\Gab_u$ and the one --- $\MV$ --- of MV-algebras was established. On the other hand, we already discussed the relationship between Abelian $\ell$-groups and idempotent semifields in Example~\ref{l-group}.

In this section we shall see how semimodules over idempotent semifields with a distinguished strong order unit (henceforth we shall call them \emph{idempotent $u$-semifields}) are related to the ones over the semiring reducts of MV-algebras. Even if so far we have preferred the $\vee$-$\odot$ notation for MV-semirings, in this section we shall use the $\wedge$-$\oplus$ semiring reduct of MV-algebras for simplicity. Indeed, the results and constructions we are going to present in the present section are better presented and more clearly visualized using the $\wedge$-$\oplus$ notation; on the other hand, according to Remark \ref{star}, the two notations are equivalent up to the application of $^*$.

The main ingredients of this section are some standard constructions and results involving tensor products of semimodules. The definition we recall below and all the results from Theorem \ref{tensormvsemi} to Theorem \ref{adjfunct} can be also derived as instances of some more general ones presented in \cite{kat2,kat4,kat7}; however, technical reasons suggest us to present them along with their specific proofs.
\begin{definition}\label{bimor}
Let $S$ be a semiring, $M$ a right $S$-semimodule, $N$ a left $S$-semimodule and $L$ a commutative monoid. A map $f: M \times N \lto L$ is called an \emph{$S$-bimorphism} if, for all $x, x_1, x_2 \in M$, $y, y_1, y_2 \in N$, $a \in S$, the following conditions hold:
\begin{enumerate}[(i)]
\item $f(x_1 + x_2, y) = f(x_1,y) + f(x_2,y)$,
\item $f(x, y_1 + y_2) = f(x,y_1) + f(x,y_2)$,
\item $f(xa,y) = f(x,ay)$.
\end{enumerate}
The \emph{$S$-tensor product} $M \tensor_S N$ is the codomain of the universal bimorphism with domain $M \times N$.
\end{definition}
It is known (see~\cite{golan}) how to construct the tensor product of a right and a left semimodule over the same semiring: it is the quotient, under a suitable monoid congruence, of the free commutative monoid over the Cartesian product of the two given semimodules. In \cite{lms} an alternative construction for semimodules over commmutative idempotent semirings --- which is proved to be equivalent \cite[Section 4, Theorem 1]{lms} --- is presented together with several results about the preservation of order completeness under tensor product.

Obviously, the constructions of \cite{golan,kat2,kat4,kat7} hold also in the case of semimodules on idempotent semirings, and the one presented in \cite{lms} would apply to MV-semirings since they are commutative and idempotent; nonetheless, since semimodules over idempotent semirings are semilattices with identity, it is possible to show yet another equivalent construction of their tensor products (see also Section 2 and Definition 3.1 of \cite{kat4}). Before we present it we recall that, given a set $X$, the free semilattice with identity over $X$ is up to isomorphisms the set $\wp_F(X)$ of the finite subsets of $X$ equipped with set-theoretic union and the bottom element $\varnothing$.

\begin{theorem}\label{tensormvsemi}
Let $A$ be an idempotent semiring, $M$ a right $A$-semimodule and $N$ a left $A$-semimodule. Then the tensor product $M \tensor_A N$ is, up to isomorphisms, the quotient $\wp_F(M \times N)/\equiv_R$ of the free semilattice generated by $M \times N$ with respect to the semilattice congruence generated by the set $R$:
\begin{equation}\label{R}
R = \left\{
	\begin{array}{l}
		\left(\left\{\left(\bigvee X, y\right)\right\}, \bigcup_{x \in X}\{(x,y)\}\right) \\
		\left(\left\{\left(x, \bigvee Y\right)\right\}, \bigcup_{y \in Y}\{(x,y)\}\right) \\
		\left(\{(x a, y)\}, \{(x,a y)\}\right) \\
	\end{array} \right\vert
	\left.
	\begin{array}{l}
	X \in \wp_F(M), y \in N \\
	Y \in \wp_F(N), x \in M \\
	a \in A \\
	\end{array}
	\right\}.
\end{equation}
\end{theorem}
\begin{proof}
Let $L = \la L, +, 0_L\ra$ be any commutative monoid and $f: M \times N \lto L$ be an $A$-bimorphism. Since $M$ and $N$ are semilattices, the image of $f$ is an idempotent submonoid of $L$, i.~e. a semilattice; hence we can extend $f$ to a monoid homomorphism $h_f: \wp_F(M \times N) \lto L$ with domain the free semilattice over $M \times N$; thus $h_f \circ \sigma = f$, where $\sigma: (x,y) \in M \times N \lmapsto \{(x,y)\} \in \wp_F(M \times N)$ is the singleton map. On the other hand, the fact that $f$ is an $A$-bimorphism implies $f\left(\bigvee X, y\right) = \sum_{x \in X} f(x,y)$, $f\left(x,\bigvee Y\right) = \sum_{y \in Y} f(x,y)$ and $f(x a, y) = f(x, a y)$, for all $x \in M$, $y \in M$, $X \in \wp_F(M)$, $Y \in \wp_F(N)$ and $a \in A$. Then, since $h_f$ acts as a semilattice homomorphism, we have $h_f\left(\left\{\left(\bigvee X,y\right)\right\}\right) = h_f\left(\bigcup_{x \in X}\{(x,y)\}\right)$ and $h_f\left(\left\{\left(x,\bigvee Y\right)\right\}\right) = h_f\left(\bigcup_{y \in Y}\{(x,y)\}\right)$. Moreover, we have
$$\begin{array}{l}
h_f(\{(x a,y)\}) = (h_f \circ \sigma)(x a,y) = f(x a,y) \\
= f(x,a y) = (h_f \circ \sigma)(x,a y) = h_f(\{(x,a y)\}).
\end{array}$$

The above means that the kernel of $h_f$ contains $R$. Let $P$ denote the quotient semilattice $\wp_F(M \times N)/\equiv_{R}$ and let $\pi$ be the canonical projection of $\wp_F(M \times N)$ onto $P$. Then the map
$$k_f \ : \quad [X]_{\equiv_{R}} \ \in \ P \quad \lmapsto \quad h_f(X) \ \in \ L$$
is a well-defined monoid homomorphism. Moreover we have $k_f \circ \pi \circ \sigma = h_f \circ \sigma = f$, so we have extended the $A$-bimorphism $f$ to a monoid homomorphism $k_f$, and it is easy to verify that the map $\tau = \pi \circ \sigma$ from $M \times N$ to $P$ is indeed an $A$-bimorphism.

The following commutative diagram should clarify the constructions above.
$$\xymatrix{
M \times N \ar[rr]^\sigma \ar[rddd]_f \ar[rd]_{\tau} && \wp_F(M \times N) \ar[lddd]^{h_f} \ar[ld]^{\pi} \\
 & P  \ar[dd]^{k_f} & \\
 && \\
 & L & \\
}$$

It is self-evident that $P$ and $\tau$ do not depend either on the monoid $L$ or on the $A$-bimorphism $f$. Therefore $\tau$ is the universal $A$-bimorphism whose domain is $M \times N$, and $P$ is its codomain, i.~e. $P \cong M \tensor_A N$.
\end{proof}

If $x \in M$ and $y \in N$, we will denote by $x \tensor y$ the image of the pair $(x,y)$ under $\tau$, i.~e. the congruence class $[\{(x,y)\}]_{\equiv_{R}}$, and we will call it an \emph{$A$-tensor} or, if there will not be danger of confusion, simply a \emph{tensor}. It is clear, then, that every element of $M \tensor_A N$ is the join of a finite set of tensors, so
$$M \tensor_A N = \left\{\bigvee_{i=1}^n x_i \tensor y_i \ \Big\vert \ x_i \in M, y_i \in N, n \in \N \right\}.$$

Let now $A$ and $B$ be two idempotent semirings, if $M$ is a $B$-$A$-bisemimodule and $N$ is a left $A$-semimodule, then the tensor product $M \tensor_A N$ naturally inherits a structure of left $B$-semimodule from the one defined on $M$:
\begin{equation*}
\star_l: \ \left(b, \bigvee_{i=1}^n x_i \tensor y_i\right) \in B \times \left(M \tensor N\right) \ \lto \ \bigvee_{i=1}^n (b \cdot x_i) \tensor y_i \in M \tensor N.
\end{equation*}
Indeed it is trivial that $\star$ distributes over finite joins in both coordinates; on the other hand, the external associative law comes straightforwardly from the fact that $M$ is a left $B$-semimodule. Analogously, if $M$ is a right $A$-semimodule and $N$ is an $A$-$B$-bisemimodule, then the tensor product $M \tensor_A N$ is a right $B$-semimodule with the scalar multiplication defined, obviously, as
\begin{equation*}
\star_r: \ \left(\bigvee_{i=1}^n x_i \tensor y_i, b\right) \in \left(M \tensor N\right) \times B \ \lto \ \bigvee_{i=1}^n x_i \tensor (y_i \cdot b) \in M \tensor N.
\end{equation*}
Therefore, it also follows that, if $C$ is another idempotent semiring such that $M$ is a $B$-$A$-bisemimodule and $N$ is an $A$-$C$-bisemimodule, then $M \tensor_A N$ is a $B$-$C$-bisemimodule. In particular, if $A$ is a commutative idempotent semiring, any tensor product of $A$-semimodules is an $A$-semimodule itself.

All these properties of the tensor product of semimodules will allow us to show some relations between tensor products and hom-sets. First we need the following lemma.

\begin{lemma}\label{diam}
Let $A$ and $B$ be idempotent semirings, $M$ an $A$-$B$-bisemimodule and $N$ a left $A$-semimodule. Then $\hom_A(M,N)$ is a left $B$-semimodule with the external product $\bullet_l$ defined, for $b \in B$, $h \in \hom_A(M,N)$ and $x \in M$, by
\begin{equation*}
(b \bullet_l h)(x) = h(x \cdot_{B} b),
\end{equation*}
$\cdot_{B}$ denoting the right external product of $M$.

Analogously, if $M$ is a $B$-$A$-bisemimodule and $N$ is a right $A$-semimodule, then $\hom_A(M,N)$ is a right $B$-semimodule with the external product $\bullet_r$ defined, for $b \in B$, $h \in \hom_A(M,N)$ and $x \in M$, by
\begin{equation}\label{homrightq}
(h \bullet_r b)(x) = h(b \cdot_{B} x),
\end{equation}
$\cdot_{B}$ denoting the left external product of $M$.
\end{lemma}
\begin{proof}
We will consider only the first case, the latter being completely analogous.

Given a scalar $b \in B$ and an $A$-semimodule homomorphism $h$, it is immediate to verify that the map $b \bullet_l h$ sends the bottom element to the bottom element and preserves finite joins. The fact that $b \bullet_l h$ preserves also the right multiplication comes from the fact that $M$ is bisemimodule; indeed, for any $a \in A$ and $x \in M$, we have
$$\begin{array}{l}
(b \bullet_l h)(a \cdot_A x) \\
= h((a \cdot_A x) \cdot_{B} b) = h(a \cdot_A (x \cdot_{B} b)) = a \cdot_A h(x \cdot_{B} b) \\
= a \cdot ((b \bullet_l h)(x)).
\end{array}$$
Therefore $b \bullet_l h \in \hom_A(M,N)$ for all $b \in B$ and $h \in \hom_A(M,N)$. Now let $b, b' \in B$, $h, h' \in \hom_A(M,N)$, and $x \in M$. Conditions (SM4) and (SM5) of Definition \ref{semimodule} are obviously verified while, for (SM1--SM3), we have:
\begin{itemize}
\item[-] $(b \bullet_l (b' \bullet_l h))(x) = h((x \cdot_B b) \cdot_B b') = h(x \cdot_B (bb')) = ((bb') \bullet_l h)(x)$;
\item[-] $(b \bullet_l (h \vee h'))(x) = h(x \cdot_B b) \vee h'(x \cdot_B b) = ((b \bullet_l h) \vee (b \bullet_l h'))(x)$;
\item[-] $((b \vee b') \bullet_l h)(x) = h(x \cdot_B (b \vee b')) = h(x \cdot_B b) \vee h(x \cdot_B b') = ((b \bullet_l h) \vee (b' \bullet_l h))(x)$.
\end{itemize}
The assertion is proved.
\end{proof}

\begin{theorem}\label{isohomtensmq}
Let $A$ and $B$ be idempotent semirings, $M$ a right $A$-semimodule, $N$ an $A$-$B$-bisemimodule and $P$ a right $B$-semimodule. Then, if we consider the right $B$-semimodule $M \tensor_A N$ and the right $A$-semimodule $\hom_{B}(N,P)$, we have
$$\hom_{B}(M \tensor_A N, P) \cong_\sL \hom_A(M,\hom_{B}(N,P)).$$
\end{theorem}
\begin{proof}
For any $B$-semimodule homomorphism $h$ from $M \tensor_A N$ to $P$ and any $x \in M$, $h$ defines a map
$$h_x: y \in N \lmapsto h(x \tensor y) \in P.$$
Since $h$ is a $B$-semimodule homomorphism, for all $y, y_1, y_2 \in N$ and $b \in B$, we have
$$\begin{array}{l}
h_x\left(y_1 \vee y_2\right) \\
= h\left(x \tensor (y_1 \vee y_2\right) \\
= h\left((x \tensor y_1) \vee (x \tensor y_2)\right) \\
= h(x \tensor y_1) \vee h(x \tensor y_2) \\
= h_x(y_1) \vee h_x(y_2)
\end{array}$$
and
$$h_x(y \cdot b) = h(x \tensor (y \cdot b)) = h((x \tensor y) \cdot b) = h(x \tensor y) \cdot b = h_x(y) \cdot b,$$
so $h_x \in \hom_{B}(N,P)$, for any fixed $x$. Hence we have a map $h_-: x \in M \lmapsto h_x \in \hom_{B}(N,P)$. Since $h$ is also a semimodule homomorphism it is, in particular, also a semilattice homomorphism and therefore $h_-$ is a semilattice homomorphism as well:
$$\begin{array}{l}
h_{x_1 \vee x_2}(y) \\
= h\left(\left(x_1 \vee x_2\right) \tensor y\right) \\
= h\left((x_1 \tensor y) \vee (x_2 \tensor y)\right) \\
= h(x_1 \tensor y) \vee h(x_2 \tensor y) \\
= h_{x_2}(y) \vee h_{x_2}(y)
\end{array},$$
for all $x_1, x_2 \in M$ and $y \in N$. Moreover, if $a \in A$, by (\ref{R}) and (\ref{homrightq}),
$$h_{x \cdot a}(y) = h((x \cdot a) \tensor y) = h(x \tensor (a \cdot y)) = h_x(a \cdot y) = (h_x \bullet_r a)(y),$$
for all $x \in M$, $y \in N$, $a \in A$, so $h_-$ is an $A$-semimodule homomorphism. Besides, we also have
$$\left(h \vee g\right)(x \tensor y) = h(x \tensor y) \vee g(x \tensor y) = h_x(y) \vee g_x(y) = \left(h_x \vee g_x\right)(y),$$
for any $h, g \in \hom_{B}(M \tensor_A N, P)$, for all $x \in M$, for all $y \in N$.

Therefore we have a semilattice homomorphism
\begin{equation*}
\zeta: \ \hom_{B}(M \tensor_A N, P) \ \lto \ \hom_A(M, \hom_{B}(N,P)),
\end{equation*}
defined by $\zeta(h) = h_-$.

Let us show that $\zeta$ has an inverse. If $f \in \hom_A(M,\hom_{B}(N,P))$, then the map $f': (x,y) \in M \times N \lmapsto (f(x))(y) \in P$ is clearly an $A$-bimorphism. Hence there exists a unique homomorphism $h_{f'}: M \tensor_A N \lto P$ such that $h_{f'}(x \tensor y) = f'(x,y) = (f(x))(y)$, for all $x \in M$ and $y \in N$, and clearly $\zeta(h_{f'}) = f$. On the other hand, if $f = \zeta(h)$ with $f \in \hom_A(M,\hom_{B}(N,P)$ and $h \in \hom_{B}(M \tensor_A N, P)$, then the uniqueness of the homomorphism extending the map $f'$ to $M \tensor_A N$ ensures us that $h_{f'} = h$. Then we have the inverse semilattice homomorphism
$$\zeta^{-1}: \ f \in \hom_A(M,\hom_{B}(N,P)) \ \lmapsto \ h_{f'} \in \hom_{B}(M \tensor_A N, P),$$
and the theorem is proved.
\end{proof}

With a completely analogous proof, we have
\begin{theorem}\label{isohomtensmq'}
Let $A$ and $B$ be idempotent semirings, $M$ a $B$-$A$-bisemimodule, $N$ a left $A$-semimodule and $P$ a left $B$-semimodule. Then, if we consider the left $B$-semimodule $M \tensor_A N$ and the left $A$-semimodule $\hom_{B}(M,P)$, we have
$$\hom_{B}(M \tensor_A N, P) \cong_\sL \hom_A(N,\hom_{B}(M,P)).$$
\end{theorem}

\begin{lemma}\label{homqm}
Let $A$ be an idempotent semiring and $M$ be an $A$-semimodule. Then, considering $A$ as a left $A$-semimodule, we have
\begin{equation*}
\hom_A(A,M) \cong_\sL M.
\end{equation*}
\end{lemma}
\begin{proof}
First of all we observe that, for any fixed $x \in M$, the map $g_x: a \in A \lto a \cdot x \in M$ is trivially an $A$-semimodule homomorphism. Then we can consider the map $\varphi: x \in M \lto g_x \in \hom_A(A,M)$, which is clearly a semilattice homomorphism.

Let us consider also the map $\psi: f \in \hom_A(A,M) \lto f(1) \in M$. Again, it is immediate to verify that $\psi$ is a semilattice homomorphism. But we also have:
$$((\varphi \circ \psi)(f))(a) = (\varphi(f(1)))(a) = g_{f(1)}(a) = a \cdot f(1) = f(a),$$
for all $f \in \hom_A(A,M)$ and $a \in A$, and
$$(\psi \circ \varphi)(x) = \psi(g_x) = g_x(1) = 1 \cdot x = x,$$
for all $x \in M$.

Thus $\varphi \circ \psi = \id_{\hom_A(A,M)}$ and $\psi \circ \varphi = \id_M$, i.~e. $\varphi$ is an isomorphism whose inverse is $\psi$, and the thesis follows.
\end{proof}

As a consequence of the previous result, the $A$-semimodule structure defined on $\hom_A(A,M)$ by Lemma~\ref{diam} is isomorphic to $M$.

\begin{lemma}\label{indmod}
Let $A$ and $B$ be idempotent semirings and $h: A \lto B$ a semiring homomorphism. Then $h$ induces a structure of $A$-semimodule on any $B$-semimodule.

In particular, $h$ induces structures of $A$-bisemimodule, $B$-$A$-bisemimodule and $A$-$B$-bisemimodule on $B$ itself.
\end{lemma}
\begin{proof}
Let $N$ be a $B$-semimodule with scalar multiplication $\cdot$. It is easy to verify that
\begin{equation}\label{starh}
\cdot_h: (a, x) \in A \times N \lmapsto h(a) \cdot x \in N
\end{equation}
makes $N$ into an $A$-semimodule, henceforth denoted by $N_h$. Since $B$ is a bisemimodule over itself, the second part of the assertion follows immediately.
\end{proof}

The operation performed in (\ref{starh}) is well-known in the theory of ring modules as \emph{restricting the scalars along $h$}. In fact it defines a functor
\begin{equation}\label{subh}
\begin{array}{cccc}
H: & B\Mod & \lto & A\Mod \\
& N & \lmapsto & N_h
\end{array}
\end{equation}
having both a right and a left adjoint, as shown by the following result.

\begin{theorem}\label{adjfunct}
The functor $H$ defined in (\ref{subh}) is an embedding and has both a left and a right adjoint.
\end{theorem}
\begin{proof}
The fact that $H$ is an embedding, namely, that $H$ is injective on morphisms is obvious. Indeed $H$ neither affects the underlying set of each object, nor the underlying map of each morphism, hence different morphisms in $B\Mod$ are sent to different morphisms in $A\Mod$ by $H$.

Now, for any $M \in A\Mod$, viewing $B$ as a $B$-$A$-bisemimodule, we can construct the tensor product $B \tensor_{A} M$ which is a left $B$-semimodule. We claim that
\begin{equation*}
\begin{array}{cccc}
H_l: & A\Mod & \lto & B\Mod \\
			& M			& \lmapsto & B \tensor_{A} M
\end{array}
\end{equation*}
is the left adjoint of $H$. In order to prove our claim, we need to show that, for any $A$-semimodule $M$ and any $B$-semimodule $N$, there exists a natural bijection between $\hom_{B}(B \tensor_{A} M, N)$ and $\hom_A(M,N_h)$. The first hom-set is isomorphic, as a semilattice, to $\hom_A(M, \hom_{B}(B,N))$, by Theorem~\ref{isohomtensmq'}; on the other hand, by Lemma~\ref{homqm}, $\hom_{B}(B,N) \cong_{\sL} N$ and such an isomorphism is an $A$-semimodule isomorphism (with $N_h$ instead of $N$) for how the $A$-semimodule structure is induced on $\hom_{B}(B,N)$. Hence the two hom-sets are isomorphic semilattices, and $H_l$ is the left adjoint of $H$. 

The right adjoint of $H$ is defined by
\begin{equation*}
\begin{array}{cccc}
H_r: & A\Mod & \lto & B\Mod \\
			 & M & \lmapsto & \hom_A(B_h,M),
\end{array}
\end{equation*}
where the left $B$-semimodule structure on $\hom_A(B_h,M)$ is the one introduced in Lemma~\ref{diam}. This part of the proof is analogous to the case of $H_l$. Indeed, for any $A$-semimodule $M$ and any $B$-semimodule $N$, by Theorem~\ref{isohomtensmq'}, $\hom_{B}(N,H_r(M))$ --- namely $\hom_{B}(N, \hom_A(B_h,M))$ --- is isomorphic, as a semilattice, to $\hom_A((B \tensor_{B} N)_h, M)$; on the other hand, since every tensor $b \tensor y \in B \tensor_{B} N$ can be rewritten in the form $1 \tensor b \cdot y$, such a tensor product is easily seen to be isomorphic to $N_h$. Therefore $\hom_{B}(N,H_r(M))$ is a semilattice isomorphic to $\hom_A(N_h, M)$ and the theorem is proved.
\end{proof}

We recall the following category-theoretic concepts from \cite[Definitions 6.1, 6.5]{cats}.
\begin{definition}
Let $\cat C$ and $\cat C'$ be categories, and $F,G: \cat C \lto \cat C'$ be functors. A \emph{natural transformation} $\tau$ from $F$ to $G$ is a function that assigns to each $\cat C$-object $X$ a $\cat C'$-morphism $\tau_X: FX \lto GX$ in such a way that, for each $\cat C$-morphism $f: X \lto Y$, $Gf \circ \tau_X = \tau_Y \circ Ff$.

A natural transformation $\tau$ such that $\tau_X$ is a $\cat C'$-isomorphism for each $\cat C$-object $X$ is called a \emph{natural isomorphism}.
\end{definition}

\begin{theorem}\label{emb}
Let $A$ and $B$ be idempotent semirings and $h: A \lto B$ an onto semiring homomorphism. Then the functor $H$ defined in (\ref{subh}) is a full embedding.

Moreover the left adjoint $H_l$ is, up to a natural isomorphism, the left inverse of $H$, that is, $H_l \circ H$ and the identity functor $\ID_{B\Mod}$ are naturally isomorphic.
\end{theorem}
\begin{proof}
We want to prove that the hom-set restrictions of $H$ are surjective, namely, that for any $M, N \in B\Mod$ and $g \in \hom_A(M_h,N_h)$ there exists $f \in \hom_B(M,N)$ such that $Ff = g$. Actually, since $Ff = f$ for all $f \in \hom_B(M,N)$, what we need to prove is that $g$ is also a $B$-semimodule homomorphism and, since $g$ is in particular a semilattice homomorphism, we just need to prove that the it preserves the scalar multiplication with coefficients in $B$. So let $b \in B$; by the surjectivity of $h$ there exists $a \in A$ such that $h(a) = b$. Hence $g(b \cdot x) = g(h(a) \cdot x) = g(a \cdot_h x) = a \cdot_h g(x) = h(a) \cdot g(x) = b \cdot g(x)$, for all $x \in M$, and therefore $g \in \hom_B(M,N)$.

It is immediate to verify that, under the hypothesis that $h$ is onto, the map $x \in M_h \lmapsto e \tensor x \in B \tensor_A M_h$ is a $B$-semimodule isomorphism for all $M \in B\Mod$, hence $H_l \circ H$ is naturally isomorphic to the identity functor $\ID_{B\Mod}$.
\end{proof}

Before showing an interesting application of the previous results of this section to MV-semimodules, we recall that the category $\MV$ of MV-algebras, with MV-algebra homomorphisms, is equivalent to $\ell\Gab_u$, namely, the category of lattice-ordered Abelian groups with a distinguished strong order unit whose morphisms are lattice-ordered group homomorphisms that preserve the distinguished strong unit \cite{mun}. The two functors that form such an equivalence are usually denoted by $\Gamma: \ell\Gab_u \lto \MV$ and $\Xi: \MV \lto \ell\Gab_u$; while the former is very easy to present (basically it is the construction presented in Example \ref{0u} where $u$ is the distinguished strong unit) and shall be recall hereafter, the latter requires more work and the details of its construction are not really relevant to our discussion. However, a detailed yet relatively concise presentation of Mundici categorical equivalence is presented in \cite[Chapter 2]{mvbook}. 

Let $\la G, +, -, 0, \vee, \wedge, u \ra$ be an Abelian $u\ell$-group with distinguished strong order unit $u$. Then the MV-algebra $\Gamma(G)$ is $\la [0,u], \oplus, ^*, 0\ra$ with $x \oplus y \bydef (x + y) \wedge u$ and $x^* \bydef u - x$ for all $x, y \in [0,u]$. The mapping $\Gamma: G \in \ell\Gab_u \lmapsto \Gamma(G) \in \MV$ is a full, faithful and isomorphism-dense functor.

As we observed in Example \ref{l-group} the category $\ell\Gab$ is isomorphic to the one of idempotent semifields; therefore, for any idempotent $u$-semifield $\la F, \wedge, +, -, \top, 0, u \ra$, we obtain an MV-algebra by applying the $\Gamma$ functor to the Abelian $u\ell$-group $\la F\setminus\{\top\}, +, -, 0, \vee, \wedge, u\ra$, with $\vee$ defined by means of $\wedge$ and $-$. 
In what follows, 
given an idempotent $u$-semifield $\la F, \wedge, +, -, \top, 0, u \ra$, with a slight abuse of notation we shall denote by $\Gamma(F)$ the MV-algebra $\la [0,u], \oplus, ^*, 0\ra$, and by $F_+$ the positive cone of $F$, namely, the set $\{x \in F \mid x \geq 0\}$ which is obviously a subsemiring of $F$. 

Now, using the results we proved so far in the present section, we will show that, for any idempotent $u$-semifield $F$, the $\Gamma$ functor induces a full embedding of the category $\Gamma(F)\Mod$ into the category $F_+\Mod$. As a first step, let us show that the functor $\Gamma$ defines a canonical onto semiring homomorphism from the positive cone of any idempotent $u$-semifield $F$ to its corresponding MV-algebra $\Gamma(F)$. 
\begin{lemma}\label{Gamma}
Let $F$ be an idempotent $u$-semifield. Then the function
\begin{equation*}
\begin{array}{cccc}
\gamma: & F_+ & \lto & \Gamma(F)\wr \\
			  & a 	& \lmapsto & a \wedge u
\end{array}
\end{equation*}
is a semiring onto homomorphism.
\end{lemma}
\begin{proof}
It is immediate to verify that $\gamma(0) = 0$ and $\gamma(\top) = u$. It is immediate to see that $\gamma$ preserves the $\wedge$ operation: for all $a,b \in F_+$, $\gamma(a \wedge b) = (a \wedge b) \wedge u = a \wedge b \wedge u \wedge u = (a \wedge u) \wedge (b \wedge u) = \gamma(a) \wedge \gamma(b)$.

For what concerns the sum, we have:
$$\begin{array}{l}
\gamma(a) \oplus \gamma(b) = ((a \wedge u) +(b \wedge u)) \wedge u \\ = ((a +(b \wedge u)) \wedge (u +(b \wedge u)) \wedge u \\ = (a + b) \wedge (a + u) \wedge (u + b) \wedge (u + u) \wedge u \\ = (a + b) \wedge u \\ = \gamma(a + b).
\end{array}$$

Last, the fact that the map $\gamma$ is surjective is obvious since $a \in \gamma^{-1}(a)$ for all $a \in [0,u]$.
\end{proof}

By Lemmas~\ref{indmod} and \ref{Gamma} and Theorem~\ref{adjfunct} the homomorphism $\gamma$ defines an adjoint and coadjoint functor
\begin{equation}\label{G}
G: \Gamma(F)\Mod \lto F_+\Mod
\end{equation}
for any idempotent $u$-semifield $F$. Combining Theorem \ref{emb} with Lemma \ref{Gamma} we obtain the following immediate result.
\begin{corollary}\label{mvemb}
The functor $G$ defined in (\ref{G}) is a full embedding and its left adjoint $G_l$ is its left inverse.
\end{corollary}

It is interesting to notice that the functor $G_l$ somehow ``truncates'' $F_+$-semimodules to $\Gamma(F)$-semimodules similarly to how $\Gamma$ truncates idempotent $u$-semifields to MV-algebras. We explain this statement starting from free semimodules.

Let $F$ be an idempotent $u$-semifield, $A = \Gamma(F)$ and $F_+^{(X)}$ be the free $F_+$-se\-mi\-mod\-ule over a given set $X$. Moreover, let us denote by $\chi_x$ and $\chi_x'$ the maps defined, respectively, in (\ref{chimv}) and (\ref{chisf}).

Let us consider the function
$$f: \ (a, \alpha) \in A \times F_+^{(X)} \ \lmapsto \ a \oplus \bigwedge_{x \in \supp\alpha} \gamma(\alpha(x)) \oplus \chi_x' \in A^{(X)},$$
and let $\alpha, \alpha' \in F_+^{(X)}$ and $a,a'  \in A$. We have:
$$\begin{array}{l}
f(a \wedge a', \alpha) \\
= (a \wedge a') \oplus \left(\bigwedge\limits_{x \in \supp\alpha} \gamma(\alpha(x)) \oplus \chi_x'\right) \\
= \left(a \oplus \bigwedge\limits_{x \in \supp\alpha} \gamma(\alpha(x)) \oplus \chi_x'\right) \wedge \left(a' \oplus \bigwedge\limits_{x \in \supp\alpha} \gamma(\alpha(x)) \oplus \chi_x'\right) \\
= f(a, \alpha) \wedge f(a', \alpha),
\end{array}$$
similarly $f(a, \alpha \wedge \alpha') = f(a,\alpha) \wedge f(a,\alpha')$. Now let $b \in F_+$; if $b \neq \top$ then $\supp \alpha = \supp(b + \alpha)$ and we have
$$\begin{array}{l}
f(a, b + \alpha) \\
= a \oplus \bigwedge\limits_{x \in \supp\alpha)} \gamma(b+\alpha(x)) \oplus \chi_x' \\
= a \oplus \bigwedge\limits_{x \in \supp\alpha)} \gamma(b) \oplus \gamma(\alpha(x)) \oplus \chi_x' \\
= a \oplus \gamma(b) \oplus \bigwedge\limits_{x \in \supp\alpha} \gamma(\alpha(x)) \oplus \chi_x' \\
= f(a \oplus \gamma(b), \alpha).
\end{array}$$
If $b = \top$ then $f(a, \top + \alpha) = a \oplus \u = \u = f(a \oplus \u, \alpha) = f(a \oplus \g(\top), \alpha)$.

So $f$ is an ${F_+}$-bimorphism (see Definition \ref{bimor}), hence it defines a semilattice homomorphism $\phi: A \tensor_{F_+} {F_+}^{(X)} \lto A^{(X)}$ which is actually an $A$-semimodule homomorphism for the commutativity of $A$. Let us now consider the map
$$\begin{array}{llll}
\psi: & A^{(X)} & \lto 		 & A \tensor_{F_+} {F_+}^{(X)} \\
			& \alpha	& \lmapsto & 0 \tensor \bigwedge\limits_{x \in \supp\alpha} \alpha(x) + \chi_x
\end{array}.$$
It is self-evident that $\phi \circ \psi = \id_{A^{(X)}}$; on the other hand, for any tensor $a \tensor \alpha \in A \tensor_{F_+} {F_+}^{(X)}$,
$$\begin{array}{l}
a \tensor \alpha = a \tensor \left(\bigwedge\limits_{x \in \supp\alpha} \alpha(x) + \chi_x\right) \\
= \bigwedge\limits_{x \in \supp\alpha} (a \tensor  (\alpha(x) + \chi_x)) \\
= \bigwedge\limits_{x \in \supp\alpha} ((a \cdot_\g \alpha(x)) \tensor \chi_x) \\
= \bigwedge\limits_{x \in \supp\alpha} ((a \oplus \g(\alpha(x))) \tensor \chi_x)
\end{array}$$
and
$$\begin{array}{l}
(\psi \circ \phi)(a \tensor \alpha) = (\psi \circ \phi)\left(\bigwedge\limits_{x \in \supp\alpha}(a \oplus \gamma(\alpha(x))) \tensor \chi_x\right) \\
= \bigwedge\limits_{x \in \supp\alpha}\left(a \oplus \gamma(\alpha(x)) \oplus (\psi \circ \phi)(0 \tensor \chi_x)\right) \\
= \bigwedge\limits_{x \in \supp\alpha} \left(a \oplus \gamma(\alpha(x)) \oplus \psi(0 \tensor \chi_x')\right) \\
= \bigwedge\limits_{x \in \supp\alpha}(a \oplus \gamma(\alpha(x))) \oplus \left(0 \tensor \chi_x\right) \\
= \bigwedge\limits_{x \in \supp\alpha} (a \oplus \gamma(\alpha(x))) \tensor \chi_x,
\end{array}$$
whence $\psi \circ \phi = \id_{A \tensor_{F_+} {F_+}^{(X)}}$. It follows that $A^{(X)}$ and $A \tensor_{F_+} {F_+}^{(X)}$ are isomorphic.

In the general case, if $M$ is an ${F_+}$-semimodule and $X$ is a set of generators for it, then $M$ is homomorphic image of ${F_+}^{(X)}$, that is there exists an onto homomorphism $f: {F_+}^{(X)} \lto M$. So, as in the previous case, we can define the map $f': (a, \alpha) \in A \times {F_+}^{(X)} \lmapsto a \tensor f(\alpha) \in A \tensor_{F_+} M$ which is easily seen to be an onto ${F_+}$-bimorphism and, therefore, induces an onto $A$-semimodule homomorphism $\varphi': A \tensor_{F_+} {F_+}^{(X)} \lto A \tensor_{F_+} M$. Hence $\varphi' \circ \psi$ is an $A$-semimodule onto homomorphism and $A \tensor_{F_+} M$ turns out to be homomorphic image of the free $A$-semimodule over the same set of generators $X$ via a sort of truncation of the original morphism $f: {F_+}^{(X)} \lto M$.

Before showing an interesting application of the previous results of this section to MV-semimodules, we recall that the category $\MV$ of MV-algebras, with MV-algebra homomorphisms, is equivalent to $\ell\Gab_u$, namely, the category of lattice-ordered Abelian groups with a distinguished strong order unit whose morphisms are lattice-ordered group homomorphisms that preserve the distinguished strong unit \cite{mun}. The two functors that form such an equivalence are usually denoted by $\Gamma: \ell\Gab_u \lto \MV$ and $\Xi: \MV \lto \ell\Gab_u$; while the former is very easy to present (basically it is the construction presented in Example \ref{0u} where $u$ is the distinguished strong unit) and shall be recall hereafter, the latter requires more work and the details of its construction are not really relevant to our discussion. However, a detailed yet relatively concise presentation of Mundici categorical equivalence is presented in \cite[Chapter 2]{mvbook}. 

Let $\la G, +, -, 0, \vee, \wedge, u \ra$ be an Abelian $u\ell$-group with distinguished strong order unit $u$. Then the MV-algebra $\Gamma(G)$ is $\la [0,u], \oplus, ^*, 0\ra$ with $x \oplus y \bydef (x + y) \wedge u$ and $x^* \bydef u - x$ for all $x, y \in [0,u]$. The mapping $\Gamma: G \in \ell\Gab_u \lmapsto \Gamma(G) \in \MV$ is a full, faithful and isomorphism-dense functor.

As we observed in Example \ref{l-group} the category $\ell\Gab$ is isomorphic to the one of idempotent semifields; therefore, for any idempotent $u$-semifield $\la F, \wedge, +, -, \top, 0, u \ra$, we obtain an MV-algebra by applying the $\Gamma$ functor to the Abelian $u\ell$-group $\la F\setminus\{\top\}, +, -, 0, \vee, \wedge, u\ra$, with $\vee$ defined by means of $\wedge$ and $-$. 
In what follows, 
given an idempotent $u$-semifield $\la F, \wedge, +, -, \top, 0, u \ra$, with a slight abuse of notation we shall denote by $\Gamma(F)$ the MV-algebra $\la [0,u], \oplus, ^*, 0\ra$. 

Now, using the results we proved so far in the present section, we will show that, for any idempotent $u$-semifield $F$, the $\Gamma$ functor induces a full embedding of the category $\Gamma(F)\Mod$ into the category $F\Mod$. As a first step, let us show that the functor $\Gamma$ defines a canonical onto semiring homomorphism from any idempotent $u$-semifield to its corresponding MV-algebra. Recalling that, for an idempotent semifield $\la F, \wedge, +, -, \top, 0\ra$, the joinis defined by $x \vee y \bydef -((-x) \wedge (-y))$, we have
\begin{lemma}\label{Gamma}
Let $F$ be an idempotent $u$-semifield. Then the function
\begin{equation*}
\begin{array}{cccc}
\gamma: & F & \lto & \Gamma(F)\wr \\
			  & a 	& \lmapsto & (a \vee 0) \wedge u
\end{array}
\end{equation*}
is a semiring onto homomorphism.
\end{lemma}
\begin{proof}
It is immediate to verify that $\gamma(0) = 0$ and $\gamma(\top) = u$. Since $F$ is obtained from a lattice-ordered Abelian group with the addition of the top element, its lattice reduct is distributive, hence $\gamma$ preserves the $\wedge$ operation: for all $a,b \in F$, $\gamma(a \wedge b) = ((a \wedge b) \vee 0) \wedge u = (a \vee 0) \wedge (b \vee 0) \wedge u \wedge u = ((a \vee 0) \wedge u) \wedge ((b \vee 0) \wedge u) = \gamma(a) \wedge \gamma(b)$.

On the other hand, according to \cite{mun}, the sum in $\Gamma(F)$ is defined by $a \oplus b = (a + b) \wedge u$ for all $a, b \in [0,u]$. Therefore the compatibility of $+$ with the lattice structure of $F$ yields, for all $a,b \in F$, $\gamma(a + b) = ((a + b) \vee 0) \wedge u = ((a \vee 0) + (b \vee 0)) \wedge u \wedge u = (((a \vee 0) + (b \vee 0)) \wedge u) \wedge u = (((a \vee 0) \wedge u) + ((b \vee 0) \wedge u)) \wedge u = (\gamma(a) + \gamma(b)) \wedge u = \gamma(a) \oplus \gamma(b)$.

The fact that the map $\gamma$ is surjective is obvious since $a \in \gamma^{-1}(a)$ for all $a \in [0,u]$.
\end{proof}

By Lemmas~\ref{indmod} and \ref{Gamma} and Theorem~\ref{adjfunct} the homomorphism $\gamma$ defines an adjoint and coadjoint functor
\begin{equation}\label{G}
G: \Gamma(F)\Mod \lto F\Mod
\end{equation}
for any idempotent $u$-semifield $F$. Combining Theorem \ref{emb} with Lemma \ref{Gamma} we obtain the following immediate result.
\begin{corollary}\label{mvemb}
The functor $G$ defined in (\ref{G}) is a full embedding and its left adjoint $G_l$ is its left inverse.
\end{corollary}

It is interesting to notice that the functor $G_l$ somehow ``truncates'' $F$-semimodules to $\Gamma(F)$-semimodules similarly to how $\Gamma$ truncates idempotent $u$-semifields to MV-algebras. We explain this statement starting from free semimodules.

Let $F$ be an idempotent $u$-semifield, $A = \Gamma(F)$ and $F^{(X)}$ be the free $F$-semimodule over a given set $X$. Moreover, let us denote by $\chi_x$ and $\chi_x'$ the maps defined, respectively, in (\ref{chimv}) and (\ref{chisf}).

Let us consider the function
$$f: \ (a, \alpha) \in A \times F^{(X)} \ \lmapsto \ a \oplus \bigwedge_{x \in \supp\alpha} \gamma(\alpha(x)) \oplus \chi_x' \in A^{(X)},$$
and let $\alpha, \alpha' \in F^{(X)}$ and $a,a'  \in A$. We have:
$$\begin{array}{l}
f(a \wedge a', \alpha) \\
= (a \wedge a') \oplus \left(\bigwedge\limits_{x \in \supp\alpha} \gamma(\alpha(x)) \oplus \chi_x'\right) \\
= \left(a \oplus \bigwedge\limits_{x \in \supp\alpha} \gamma(\alpha(x)) \oplus \chi_x'\right) \wedge \left(a' \oplus \bigwedge\limits_{x \in \supp\alpha} \gamma(\alpha(x)) \oplus \chi_x'\right) \\
= f(a, \alpha) \wedge f(a', \alpha),
\end{array}$$
similarly $f(a, \alpha \wedge \alpha') = f(a,\alpha) \wedge f(a,\alpha')$. Now let $b \in F$; if $b \neq \top$ then $\supp \alpha = \supp(b + \alpha)$ and we have
$$\begin{array}{l}
f(a, b + \alpha) \\
= a \oplus \bigwedge\limits_{x \in \supp\alpha)} \gamma(b+\alpha(x)) \oplus \chi_x' \\
= a \oplus \bigwedge\limits_{x \in \supp\alpha)} \gamma(b) \oplus \gamma(\alpha(x)) \oplus \chi_x' \\
= a \oplus \gamma(b) \oplus \bigwedge\limits_{x \in \supp\alpha} \gamma(\alpha(x)) \oplus \chi_x' \\
= f(a \oplus \gamma(b), \alpha).
\end{array}$$
If $b = \top$ then $f(a, \top + \alpha) = a \oplus \u = \u = f(a \oplus \u, \alpha) = f(a \oplus \g(\top), \alpha)$.

So $f$ is an $F$-bimorphism (see Definition \ref{bimor}), hence it defines a semilattice homomorphism $\phi: A \tensor_F F^{(X)} \lto A^{(X)}$ which is actually an $A$-semimodule homomorphism for the commutativity of $A$. Let us now consider the map
$$\begin{array}{llll}
\psi: & A^{(X)} & \lto 		 & A \tensor_F F^{(X)} \\
			& \alpha	& \lmapsto & 0 \tensor \bigwedge\limits_{x \in \supp\alpha} \alpha(x) + \chi_x
\end{array}.$$
It is self-evident that $\phi \circ \psi = \id_{A^{(X)}}$; on the other hand, for any tensor $a \tensor \alpha \in A \tensor_F F^{(X)}$,
$$\begin{array}{l}
a \tensor \alpha = a \tensor \left(\bigwedge\limits_{x \in \supp\alpha} \alpha(x) + \chi_x\right) \\
= \bigwedge\limits_{x \in \supp\alpha} (a \tensor  (\alpha(x) + \chi_x)) \\
= \bigwedge\limits_{x \in \supp\alpha} ((a \cdot_\g \alpha(x)) \tensor \chi_x) \\
= \bigwedge\limits_{x \in \supp\alpha} ((a \oplus \g(\alpha(x))) \tensor \chi_x)
\end{array}$$
and
$$\begin{array}{l}
(\psi \circ \phi)(a \tensor \alpha) = (\psi \circ \phi)\left(\bigwedge\limits_{x \in \supp\alpha}(a \oplus \gamma(\alpha(x))) \tensor \chi_x\right) \\
= \bigwedge\limits_{x \in \supp\alpha}\left(a \oplus \gamma(\alpha(x)) \oplus (\psi \circ \phi)(0 \tensor \chi_x)\right) \\
= \bigwedge\limits_{x \in \supp\alpha} \left(a \oplus \gamma(\alpha(x)) \oplus \psi(0 \tensor \chi_x')\right) \\
= \bigwedge\limits_{x \in \supp\alpha}(a \oplus \gamma(\alpha(x))) \oplus \left(0 \tensor \chi_x\right) \\
= \bigwedge\limits_{x \in \supp\alpha} (a \oplus \gamma(\alpha(x))) \tensor \chi_x,
\end{array}$$
whence $\psi \circ \phi = \id_{A \tensor_F F^{(X)}}$. It follows that $A^{(X)}$ and $A \tensor_F F^{(X)}$ are isomorphic.

In the general case, if $M$ is an $F$-semimodule and $X$ is a set of generators for it, then $M$ is homomorphic image of $F^{(X)}$, that is there exists an onto homomorphism $f: F^{(X)} \lto M$. So, as in the previous case, we can define the map $f': (a, \alpha) \in A \times F^{(X)} \lmapsto a \tensor f(\alpha) \in A \tensor_F M$ which is easily seen to be an onto $F$-bimorphism and, therefore, induces an onto $A$-semimodule homomorphism $\varphi': A \tensor_F F^{(X)} \lto A \tensor_F M$. Hence $\varphi' \circ \psi$ is an $A$-semimodule onto homomorphism and $A \tensor_F M$ turns out to be homomorphic image of the free $A$-semimodule over the same set of generators $X$ via a sort of truncation of the original morphism $f: F^{(X)} \lto M$.

\section{Concluding remarks}
\label{concl}

The results presented in this work broaden the already wide variety of connections between MV-algebras and other theories. As we anticipated, our intention was mainly to establish such new links so as to open new research lines and motivations for future works on this matter. Indeed, such a semiring-theoretic perspective on MV-algebras naturally suggests many questions and ideas.

For example, one may ask if it is possible to define tropical algebraic varieties on MV-algebras as a ``truncated'' version of the ones defined on the tropical semifield of the reals. Moreover, if such a question has a positive answer, it would be reasonable to ask whether there would be any connection between such a theory and the well-established geometric theory of MV-algebras developed mainly by Aguzzoli, Mundici and Panti (see, for instance, \cite{agu,mun2,mun3,pan}). 

Another issue that naturally arises is related to the functor $K_0$ associating an Abelian group to every MV-algebra. Obviously, such a functor immediately suggests the development of an algebraic $K$-theory of MV-algebras which, however, needs to be strongly motivated, i.~e. is expected to advance the knowledge of MV-algebras rather than being a purely speculative exercise.

Besides all these possible advances, it is unquestionable that the strong tie between semiring and semimodule theories of MV-algebras and idempotent $u$-semifields --- whose common DNA lies on Mundici categorical equivalence --- is worth to be investigated as deeply as possible. As a matter of fact, such an equivalence directly relates the tropical semifield $\la \ov\R, \min, +, \infty, 0 \ra$, which is the basis for the most important concrete model of tropical geometry, with the MV-algebra $[0,1]$, that generates the variety of MV-algebras and with respect to which \L ukasiewicz propositional calculus is standard complete.

\end{document}